\documentclass[aap,MSNbibl,dvips]{arximspdf}

\doi{10.1214/10-AAP725}
\volume{21}
\issue{3}
\pubyear{2011}
\firstpage{1180}
\lastpage{1213}

\makeatletter
\renewcommand{\epsilon}{\varepsilon}
\newcommand{\eqref}[1]{(\ref{#1})}
\newcommand{\tw}{\widetilde{W}}
\newcommand{\usimp}{\mathfrak{S}}
\newcommand{\bndsimp}{\widetilde{\mathfrak{S}}}
\newcommand{\combi}[2]{\pmatrix{#1 \cr #2 }}

\newtheorem{theo}{Theorem}
\newtheorem{lemma}[theo]{Lemma}
\newtheorem{prop}[theo]{Proposition}
\newproclaim{defn}{Definition}
\newproclaim{remark}{Remark}
\newproclaim{claim}{Claim}
\newproclaim{remarks}{Remarks}
\makeatother

\begin{document}
\begin{frontmatter}

\title{Analysis of market weights under volatility-stabilized market models}
\runtitle{Volatility-stabilized markets}

\begin{aug}
\author{\fnms{Soumik} \snm{Pal}\corref{}\thanksref{t1}\ead[label=e1]{soumik@math.washington.edu}}
\runauthor{S. Pal}
\affiliation{University of Washington}
\address{Department of Mathematics\\
University of Washington\\
Seattle, Washington 98195\\
USA\\
\printead{e1}} %adresu isvedimo komanda gale!
\end{aug}
\thankstext{t1}{Supported in part by NSF Grant DMS-10-07563.}

% HISTORY:
\received{\smonth{10} \syear{2009}}
\revised{\smonth{5} \syear{2010}}

% ABSTRACT
%
\begin{abstract}
We derive the joint density of market weights, at fixed times and
suitable stopping times, of the volatility-stabilized market models
introduced by Fernholz and Karatzas in [\textit{Ann. Finan.} \textbf
{1} (2005)
149--177]. The argument rests on computing the exit density
of a collection of independent Bessel-square processes of possibly
different dimensions from the unit simplex. We show that the law of the
market weights is the same as that of the multi-allele Wright--Fisher
diffusion model, well known in population genetics. Thus, as a side
result, we furnish a novel proof of the transition density function of
the Wright--Fisher model which was originally derived by Griffiths by
bi-orthogonal series expansion.
\end{abstract}

% KEYWORDS
%

\end{frontmatter}

%s1 ###
\section{Introduction}\label{intro} The multidimensional diffusion models named
volatility-stabilized market (VSM) models were introduced by Fernholz
and Karat\-zas~\cite{ferkar05} as toy models that nevertheless reflect
some of the traits of a real-world equity market. We refer the reader
to an excellent survey article by the same authors~\cite{ferkarspt}.
These models reflect the fact that in real markets the smaller stocks
tend to have a greater volatility and a greater rate of growth than the
larger ones.

The mathematical description of the model involves a vector-valued
continuous stochastic process $X(t)=(X_1(t), X_2(t), \ldots, X_n(t)),$
where every coordinate takes nonnegative values. Their dynamics are
determined by the following stochastic differential equation (SDE)\
with a single nonnegative parameter $\delta$: for $i=1,2,\ldots,n$,
we have
%
%e1 ###
\begin{equation}\label{vsmeq}
 \hspace*{16pt} dX_i(t) = \frac{\delta}{2}S(t)\,dt + \sqrt{X_i(t)S(t)}\,dW_i(t), \qquad
S(t)=X_1(t) + \cdots+ X_n(t).
\end{equation}
The initial vector, $X(0)$, is a point in the positive quadrant of
$\mathbb{R}
^n$, which we will denote by $\mathbb{R}^{n+}$. Here, $(W_1, W_2,
\ldots,
W_n)$ is an $n$-dimensional Brownian motion. The original article by
Fernholz and Karatzas~\cite{ferkar05} parametrizes the model by
$\alpha=\delta-1,$ which is assumed to be nonnegative. Our analysis
will consider a more general class of models where the scalar $\delta$
is replaced by a vector $(\delta_1, \ldots, \delta_n)$ of
nonnegative coordinates with only the restriction $\sum_{i=1}^n \delta
_i > 1$.

The intuition behind such a modeling becomes clear from the following
consideration. Define the vector of \textit{market weights}
%
%e2 ###
\begin{equation}\label{mktwts}
\mu_i= \frac{X_i}{\sum_{j=1}^n X_j}, \qquad  i=1,2,\ldots,n.
\end{equation}
From an economic viewpoint, market weights are a measure of the
influence that the $i$th company exerts on the entire market. These
have been studied extensively in the literature; see, for example,
articles by Hashemi~\cite{hashemi00}, Ijiri and Simon \cite
{ijirisimon74}, Jovanovic~\cite{jovanovic82} and Simon and Bonini
\cite{simonbonini58}. For a probabilistic study in the context of
another interacting market model, see the article by Chatterjee and
Pal~\cite{chatpal}.

One can alternatively express the SDE~\eqref{vsmeq} (see \cite
{ferkar05}) by writing
\[
d\log X_i(t) = \frac{\delta- 1}{2\mu_i(t)}\,dt + \frac{1}{\sqrt{\mu
_i(t)}}\,dW_i(t), \qquad  i=1,2,\ldots,n,
\]
which makes some of the features of the model immediate and visually
appealing. The smaller $\mu_i$ is, the greater the drift and the
fluctuation of $\log X_i$ are. This is the primary empirical
observation that the model is designed to capture.

In this article, we answer one of the questions left open in the
articles~\cite{ferkar05} and~\cite{ferkarspt}: how can we describe
the behavior of the vector of random market weights $(\mu_1, \ldots,
\mu_n)(t)$ under the law of the VSM model? Similar problems have been
studied by Irina Goia in her thesis~\cite{goiathesis09}; see this for
a discussion of the relationship of these models with CIR models in
mathematical finance and their relevance in the bigger picture of
stochastic portfolio theory.

As a natural culmination of the theory we develop in this article, we
consider the following generalization of VSM models.

\begin{defn}
For any $n$ nonnegative parameters $(\delta_1, \ldots, \delta_n)$,
consider the solution of the stochastic differential equation
%
%e3 ###
\begin{equation}\label{vsmeq2}
d \log X_i(t) = \frac{\delta_i - 1}{2\mu_i(t)} \,dt + \frac{1}{\sqrt
{\mu_i(t)}}\,dW_i(t), \qquad  i=1,2,\ldots,n.
\end{equation}
We call the unique-in-law solution of the above equation the VSM model
with parameters $(\delta_1, \ldots, \delta_n)$ and denote it by $V(\delta_1, \ldots, \delta_n)$.
\end{defn}

As mentioned in~\cite{ferkar05}, the uniqueness in law of the above
SDE is guaranteed by results in the theory of degenerate stochastic
differential equations as developed by Bass and Perkins in~\cite{basspersde}.

A crucial observation made in~\cite{ferkar05} in analyzing the VSM
model is the connection with Bessel-square (BESQ) processes. Given a
solution of SDE~\eqref{vsmeq}, one can construct $n$ independent BESQ
processes of dimension $2\delta$ (say), $Z_1, Z_2, \ldots, Z_n$, such
that the solution $X$ is linked with $Z=(Z_1, \ldots, Z_n)$ by an
appropriate time change. Explicitly,
%
%e4 ###
\begin{eqnarray}\label{vsmbes}
X_i(t)&=& Z_i (\Lambda(t) ), \qquad  0\le t< \infty, \qquad
i=1,2,\ldots,n,\nonumber
\\[-8pt]
\\[-8pt]
\Lambda(t)&=&\frac{1}{4}\int_0^t S(u)\,du, \qquad  S(u)=X_1(u)+ \cdots+X_n(u).
\nonumber
\end{eqnarray}
A straightforward generalization of the analysis of Fernholz and
Karatzas shows that a weak solution of the system in~\eqref{vsmeq2}
can be obtained by the following mechanism. Given a solution $X$ of
$V(\delta_1, \ldots, \delta_n)$, there exist processes $Z_1, \ldots
, Z_n$ which are independent BESQ processes of respective dimensions
$2\delta_1, \ldots, 2\delta_n$ such that the time change relation
described in~\eqref{vsmbes} continues to hold.

We have the following results.
\begin{prop}\label{vsmmarket}
Let $X=(X_1, \ldots, X_n)$ have the law $V(\delta_1, \ldots, \delta
_n)$ as in~\eqref{vsmeq2}, with initial $X_i(0) = x_i \ge0$ for every
$i$. Suppose
\[
\delta_i > 0  \qquad \mbox{for all }  i    \quad \mbox{and} \quad  d= \sum_{i=1}^n
\delta_i > 1.
\]
Let $S(t)$ denote the total sum process $X_1(t)+X_2(t)+ \cdots+X_n(t)$.
Let $\varsigma_a$ be the stopping time
\[
\varsigma_a = \inf\{ t\ge0\dvtx  S(t) =a \}, \qquad
s:=\sum_{i=1}^n x_i \le a.
\]
The joint density of the market weights $\mu=(\mu_1, \ldots, \mu
_n)$ at the stopping time $\varsigma_a$ is then given by the following
expression:
%
%e5 ###
\begin{eqnarray}\label{mwtden}
\varphi_x(y) &=& ( 1 - s/a ) \sum_{m=0}^{\infty} \frac
{\Gamma(2m + d)}{m!\Gamma(m + d)} (1+ s/a )^{-2m-d} \nonumber
\\
&&{}\times\sum_{k\ge0: k_1 + \cdots+ k_n=m}
\combi{m}{k_1\cdots k_n}
\prod_{i=1}^{n} (x_i/a)^{k_i} \operatorname{Dir}(y; k+\delta),\\
 \eqntext{y_i \ge0, \mbox{ for all $i$ and } \sum_{i=1}^n y_i =1.}
\end{eqnarray}
Here, $k+\delta$ denotes the vector $(k_1+\delta_1, \ldots,
k_n+\delta_n)$ and $\operatorname{Dir}(y;\gamma)$ is the density of the Dirichlet
distribution with parameter $\gamma$ given by
%
%e6 ###
\begin{equation}\label{symdirden}
\operatorname{Dir}(y;\gamma)= \frac{{\prod_{i=1}^n\Gamma(\gamma
_i)}}{\Gamma( \sum_{i=1}^n \gamma_i )} \prod_{i=1}^n
y_i^{\gamma_i-1}, \qquad  y_i \ge0,\ \sum_{i=1}^n y_i =1.
\end{equation}
\end{prop}

As mentioned in the abstract, the analysis requires us to compute the
exit density of a collection of independent BESQ processes, of
dimensions $\delta_1, \ldots, \delta_n$, which might be of
independent interest.

\begin{prop}\label{exitderiv}
Suppose $n \ge3$ and let $Z=(Z_1, \ldots, Z_n)$ be independent BESQ
processes of respective dimensions $\theta_1, \ldots, \theta_n$, where
\[
\theta_i > 0  \qquad \mbox{for all $i$}   \quad \mbox{and} \quad  \theta_0=\sum
_{i=1}^n \theta_i > 2.
\]
Assume that, initially, $Z(0)=z=(z_1, \ldots, z_n)$, where each
$z_i\ge0$ and $S_z:=\sum_{i=1}^n z_i < 1$.
Consider the stopping time $\sigma_1$ given by
\[
\sigma_1= \inf\{ t\dvtx  \zeta(t) \ge1 \}, \qquad  \zeta
(t)=Z_1(t)+ \cdots+Z_n(t).
\]
The density of $(Z_1, Z_2, \ldots, Z_n)(\sigma_1)$ is then given by
%
%e7 ###
\begin{eqnarray}\label{exitdensity}
\varphi_z(y) &=& ( 1 - S_z ) \sum_{m=0}^{\infty} \frac
{\Gamma(2m + \theta_0/2)}{m!\Gamma(m + \theta_0/2)} (1+ S_z
)^{-2m-\theta_0/2} \nonumber
\\
&&\hphantom{( 1 - S_z ) \sum_{m=0}^{\infty}}
{}\times\sum_{k\ge0: k_1 + \cdots+ k_n=m} \combi{m}{k_1\cdots
k_n} \prod_{i=1}^{n} z_i^{k_i} \operatorname{Dir}(y; k+\theta/2), \\
 \eqntext{y_i \ge0  \mbox{ for all $i$ and } \sum_{i=1}^n y_i =1.}
\end{eqnarray}
Here, $k+\theta/2$ denotes the vector $(k_1+\theta_1/2, \ldots,
k_n+\theta_n/2)$.

Since each $\theta_i$ is assumed to be strictly positive, the above
expression is also the exit density of the $Z$ process from the unit
simplex $\{ x\in\mathbb{R}^n\dvtx\break  x_i \ge0, \sum_{i=1}^n x_i \le1 \}$.
\end{prop}

A deeper analysis can be undertaken by noting, as we will show in
Section~\ref{polar}, that the distribution of market weights under the
VSM model is nothing but the multi-allele Wright--Fisher diffusion
model studied in population genetics. A short introduction to this
well-known and important model is provided in Section~\ref{models}.

\begin{prop}\label{vsmproof}
The process of market weights $(\mu_1, \ldots, \mu_n)$ under
$V(\delta_1, \break \ldots, \delta_n)$ is itself a diffusion, independent
of the total sum process $S$. Its law is the same as that of a
multi-allele Wright--Fisher model with mutation parameters $(\delta_1,
\ldots, \delta_n)$.\vadjust{\goodbreak}

Under the additional assumption that each $\delta_i$ is strictly
positive, the unique reversible invariant probability law for the
market weights under $V(\delta_1, \ldots, \delta_n)$ is given by the
multivariate Dirichlet distribution with parameters $(\delta_1, \ldots
, \delta_n)$.
\end{prop}

Finally, we prove a transition density formula for the market weights.
Since we show that the market weights have the same law as the
Wright--Fisher diffusions, it follows that this is the same as the
transition density for the Wright--Fisher model which was originally
derived by Griffiths in 1979~\cite{griffiths79b}; see also Griffiths
\cite{griffiths79a}. Our proof is novel and follows easily from
Proposition~\ref{vsmmarket} and suitably changing time.

\begin{prop}\label{tranmarket}
Let $p(t,\xi,y)$ denote the transition density from an initial point
$\xi$ to a final point $y$ of the market weights under the VSM model
which satisfies the same assumptions as in Proposition \ref
{vsmmarket}. Then, $p(t,\xi,y)$ is given by the formula
%
%e8 ###
\begin{eqnarray}\label{ptxiy}
p(t,\xi,y)&=& \sum_{m=0}^{\infty} \frac{\Gamma(2m + d)}{m!\Gamma(m
+ d)} b_m(t) \nonumber
\\
&&\hphantom{\sum_{m=0}^{\infty}}
{}\times \sum_{ k\ge0: k_1 + \cdots+ k_n=m} \combi{m}{k_1\cdots
k_n} \prod_{i=1}^{n} (\xi_i)^{k_i} \operatorname{Dir}(y; k+\delta
),\\
\eqntext{\xi_i \ge0,\ y_i \ge0,\ \sum_{i=1}^n \xi_i = \sum
_{i=1}^n y_i =1.}
\end{eqnarray}
The coefficients $b_m(\cdot)$ can be expressed by the Laplace
transform formula which holds for all positive $\rho$:
%
%e9 ###
\begin{eqnarray}\label{Laplaceforbm}
&& \quad \int_0^{\infty} b_m(t) t^{-3/2} e^{-\gamma^2 t/2} \exp\biggl(
-\frac{\rho^2}{2t} \biggr)\,dt\nonumber
\\[-8pt]
\\[-8pt]
&& \quad  \qquad =\sqrt{2\pi}\rho^{-1}e^{-(m+\gamma)\rho}  ( 1 - e^{-\rho}
) (1+ e^{-\rho} )^{-2m-d}, \qquad  m=0,1,2, \ldots.
\nonumber
\end{eqnarray}
Here, $\gamma=(d-1)/2$.
\end{prop}

\begin{remarks*}
(i) Tavar\'e~\cite{tavare84} gives a different proof of the
above formula for the Wright--Fisher model, where the coefficients
$b_m(t)$ are themselves linked to transition probabilities of a pure
death process in $\mathbb{Z}^+\cup\{\infty\}$. Our formula above
establishes a Laplace transform representation of the same
probabilities, which might be of some interest.

(ii) The transition density function for the Wright--Fisher
model, as derived by Griffiths, has exactly the same form for all
nonnegative values of $(\delta_1, \ldots, \delta_n)$. It should be
possible, by extending our methods, to eliminate assumptions on the
parameters. However, it is not immediate and requires further work. We
do not pursue this here since the VSM models naturally assume that
$\sum_{i=1}^n \delta_i > 1$, which corresponds to the fact that the
entire equity market never hits zero.\vadjust{\goodbreak}

(iii) There is an interest in determining whether the market
weights in equilibrium exhibit power-law decay (i.e., the $i$th largest
market weight $\mu_i$ is proportional to $i^{-\gamma}$ for some
positive $\gamma$). This is empirically observed and can be proven in
the case of certain models; see Chatterjee and Pal~\cite{chatpal} for
further motivation, references and some results involving the
Poisson--Dirichlet families of point processes with parameters $(\alpha
,0)$ where this indeed takes place. However, there does not appear to
be such a possibility for the VSM models. The finite-dimensional
invariant distributions have been identified in Proposition \ref
{vsmproof} as Dirichlet distributions. Under standard  Poisson
convergence  assumptions, the point processes of the order statistics
of Dirichlet distributions converge to Poisson--Dirichlet processes
with parameters $(0,\beta)$ for some positive $\beta$, which do not
exhibit power-law decay.
\end{remarks*}

The paper is arranged as follows. The next subsection describes the
multi-allele Wright--Fisher models and their limiting measure-valued
diffusion, the Fleming--Viot model. In Section~\ref{mktstop} we
provide proofs of Propositions~\ref{vsmmarket} and~\ref{exitderiv}.
This is achieved by defining a multidimensional functional
transformation, akin to the Kelvin transform for the Laplacian, that
utilizes inversion with respect to the unit simplex. In Section \ref
{polar} we establish the fact that the process of market weights under
the VSM model is actually the Wright--Fisher model. The analysis is
slightly generalized to include the Fleming--Viot models, which shows
the large $n$ behavior of the market weights. In Section \ref
{section:submarket} we examine the practical situation where one
considers not the entire vector of market weights, but only a subset of
it. This situation can be handled due to a recursive property of VSM models. Finally, in Section~\ref
{trandenmarket} we establish Proposition~\ref{tranmarket} as a
corollary of the previous results.

%s1.1 ###
\subsection{Notation}\label{notations} This article sometimes
requires notation that refers to similar, and yet different, objects.
To help the reader avoid confusion, we now list most of the notation
used repeatedly in the following sections.

VSM processes will be denoted throughout by $X=(X_1, \ldots, X_n),$
while BESQ processes will be written $Z=(Z_1, \ldots, Z_n)$. Their
dimensions will be the vectors $\delta$ and $\theta$, respectively.
The sum processes will be $S=\sum_{i=1}^n X_i$ and $\zeta=\sum
_{i=1}^n Z_i$, with corresponding dimensions
%
%e10 ###
\begin{equation}\label{whatistheta0}
d=\sum_{i=1}^n \delta_i  \quad \mbox{and} \quad  \theta_0=\sum
_{i=1}^n \theta_i.
\end{equation}
The stopping times $\varsigma_a$ and $\sigma_a$ denote the random
hitting times of level $a$ by the processes $S$ and $\zeta$,
respectively. It will sometimes be convenient to consider the following
transformation of the parameter $\theta$:
%
%e11 ###
\begin{equation}\label{whatisnu}
\nu_i := \theta_i/ 2 -1, \qquad  \nu_0= \sum_{i=1}^n \nu_i.\vadjust{\goodbreak}
\end{equation}

The closed positive quadrant in $n$ dimensions will be denoted by
$\mathbb{R}
^{n+}$. We denote the $n$-dimensional closed unit simplex by
%
%e12 ###
\begin{equation}\label{whatisusimp}\qquad
\usimp= \Biggl\{ x=(x_1, \ldots, x_n)\dvtx  x_i \ge0 \mbox{ for
all } i=1,\ldots,n, \mbox{ and } \sum_{i=1}^n x_i \le1 \Biggr\}.
\end{equation}
The oblique boundary of the unit simplex will be denoted by
%
%e13 ###
\begin{equation}\label{whatisbndsimp}\qquad
\bndsimp= \Biggl\{ x=(x_1, \ldots, x_n)\dvtx  x_i \ge0 \mbox{ for
all } i=1,\ldots,n, \mbox{ and } \sum_{i=1}^n x_i = 1 \Biggr\}.
\end{equation}

For any two vectors $a=(a_1, \ldots, a_n)$ and $b=(b_1, \ldots,
b_n),$ we will use the following notation:
%
%e14 ###
\begin{equation}\label{spnotn}
S_a= \sum_{i=1}^n a_i, \qquad  a^b= \prod_{i=1}^n a_i^{b_i}, \qquad  a!=
\prod_{i=1}^n a_i !.
\end{equation}

%s1.2 ###
\subsection{A brief description of various models}\label{models} In
this subsection we describe the various stochastic processes which are
all linked to VSM models.

\subsubsection{Bessel-square processes}
A comprehensive treatment of BESQ pro\-cesses can be found in the book by
Revuz and Yor~\cite{yorbook}. These one-dimensional diffusions are
indexed by a single nonnegative real parameter $\theta$ (called the
dimension) and are solutions of the stochastic differential equations
%
%e15 ###
\begin{equation}\label{besqintro}
Z(t)= x + 2 \int_0^t \sqrt{| Z(s) |}\,d\beta(s) + \theta t, \qquad
x \ge0,\ t\ge0,
\end{equation}
where $\beta$ is a one-dimensional standard Brownian motion. We denote
the law of this process by BESQ$^\theta_x$. It can be shown that the
above SDE admits a unique strong solution which remains nonnegative
throughout time.

For $\theta=1,2,3,4,\ldots$ however, the same process law can be
obtained from another perspective. It is well known that in dimension
$\theta=1,2,3,4,\ldots$ the BESQ process has the same law as that of
the square of the Euclidean norm of Brownian motion in dimension
$\theta$. The case $\theta=0$ is unique. The BESQ process for
\textit{dimension} zero is a nonnegative martingale which is a diffusion
approximation to the process of the size of the surviving population of
a critical Galton--Watson branching process.

The applications of BESQ processes, and especially of derived Bessel
processes, are too numerous to list here. As a tip of this iceberg, we
mention such diverse areas as: (i) branching process theory and
superprocesses (see Etheridge~\cite{etheridge}); (ii) Brownian path
decomposition and excursion theory (see the book by Revuz and Yor
\cite{yorbook}, Chapter XII); (iii) L\'evy processes (see the article
by Carmona, Petit and Yor~\cite{carpetiyor}); (iv) local times of
Markov processes and Dynkin's isomorphism (see Eisenbaum \cite
{eisenbaumrk}, Pitman~\cite{pitmancb} and Werner \cite
{wernerperturb}); (v) mathematical finance (see Cox, Ingersoll and Ross
\cite{cirproc}, Geman and Yor~\cite{gemanyor}); (vi) random matrices
(see Bru~\cite{bruwishart} and K\"onig and O'Connell~\cite{konigoconnell}).

\subsubsection{Wright--Fisher diffusions}
The Wright--Fisher
diffusion model (see, e.g., Ethier and Kurtz~\cite{ethkur81}, page~432) arises as the diffusion approximation of the
Wright--Fisher Markov chain model as the population size goes to
infinity. A~good source for an introduction to the biology and
mathematics of these models is Chapter 1 in the book by Durrett \cite
{durrettgenetics}.

For the purposes of this article, it is a family of diffusions with
state space $\bndsimp$ and parametrized by a vector $(\delta_1,
\ldots, \delta_n)$ of nonnegative entries. These are the solutions of
the stochastic differential equations
%e16 ###
\begin{equation}\label{whatisjacobi}
dJ(t) = \frac{1}{2} \bigl(\delta_i - d J(t) \bigr) \,dt + \tilde{\sigma}(J)\,d\beta(t), \qquad  d = \sum_{i=1}^n \delta_i.
\end{equation}
Here, $\beta$ is a standard multidimensional Brownian motion and the
diffusion matrix $\tilde{\sigma}$ is given by
%
%e17 ###
\begin{equation}\label{whatistsigma}
\tilde{\sigma}_{i,j}(x)= \sqrt{x_i} \bigl(1\{i=j\} - \sqrt{x_ix_j}
\bigr), \qquad  1\le i,j\le n.
\end{equation}
The law of this process will be denoted by $J(\delta_1, \ldots,
\delta_n)$.

In the literature this process is sometimes identified by its Markov generator:
%
%e18 ###
\begin{equation}\label{genwf}
\mathcal{A}_n = \frac{1}{2}\sum_{i,j=1}^n x_i ( 1\{i=j\} -
x_j ) \frac{\partial^2}{\partial x_i\, \partial x_j} + \sum
_{i=1}^n \frac{1}{2} (\delta_i - d x_i )\frac{\partial
}{\partial x_i}.
\end{equation}

For the case of $n=2$, the first coordinate of the Wright--Fisher
diffusion is also known as the Jacobi diffusion; see the article by
Warren and Yor~\cite{warrenyor}. Hence, the general class is sometimes
also referred to as that of multidimensional Jacobi diffusions; see,
for example, Goia~\cite{goiathesis09}.

It is known that for any $n\in\mathbb{N}$ and any strictly positive
$\delta_1, \ldots, \delta_n$, the Dirichlet distribution $\operatorname{Dir}(\delta
_1,\ldots,\delta_n)$ is the unique reversible invariant measure for
the Wright--Fisher model $J(\delta_1, \ldots, \delta_n)$; see Lemma
4.1 of~\cite{ethkur81}.

\subsubsection{Fleming--Viot diffusions}\label{section1.2.3} The large $n$ limit
of Wright--Fisher diffusions is the family of measure-valued diffusions
that are known as Fleming--Viot processes; see the survey by Ethier and
Kurtz~\cite{FVsurvey}. These diffusions take values from the set of
all probability measures on an underlying space and can be parametrized
by a linear operator. Fleming--Viot processes and Dawson--Watanabe
superprocesses are probably the most important families of
measure-valued diffusions studied in probability. For an introduction
to the rich literature in this area, see the book by Etheridge \cite
{etheridge}.

We will hardly need the general theory in this article. In fact, the
family of Fleming--Viot processes we will use has no spatial component.
Let $B$ be any Lebesgue-measurable subset of $[0,\infty)$ whose
Lebesgue measure is $\theta_0$ for some $\theta_0 \ge0$. By a
Fleming--Viot process, we refer to a stochastic process which, at any
time, takes value in the metric space of $\mathcal{P}(B)$, the set of all
probability measures supported on $B$ under the Prokhorov metric of
weak convergence. This process, say $\nu$, is defined by the following
property: for any $n=1,2,\ldots$ and any partition of $B$ into
disjoint Lebesgue-measurable sets $ A_1, \ldots, A_n$ with respective
Lebesgue measures $\theta_1, \ldots, \theta_n$, where $\theta_i \ge
0$ and $\theta_0=\sum_{i=1}^n \theta_i$, the law of the derived process
\[
( \nu(A_1), \ldots, \nu(A_n) )(t), \qquad  0\le t <
\infty,
\]
is distributed as $J(\theta_1/2, \ldots, \theta_n/2)$. We will
denote the law of the process $\nu$ by $\operatorname{FV}(B)$. We will construct such
a process later in the text, which will prove its existence. That it is
uniquely defined by the above specification is clear.

%s2 ###
\section{Description of market weights under VSM models}\label{mktstop}

Consider $n$ nonnegative parameters $(\theta_1, \ldots, \theta_n)$
and $n$ independent BESQ processes $(Z_1, \ldots,\break Z_n),$ where the
dimension of $Z_i$ is $\theta_i$ and the assumptions of Proposition
\ref{exitderiv} are satisfied. Then, as we have noted in \eqref
{vsmbes}, one can construct a process with law $V(\theta_1/2,\ldots
,\theta_n/2)$ by an appropriate time change of the BESQ processes. We
extend the notation introduced in~\eqref{vsmbes}. Recall the sum
processes $\zeta= \sum_{i=1}^n Z_i$ and $S =\sum_{i=1}^n X_i$.

The market weights at any time $t$ are then given by the relation
\[
\mu_i(t) = \frac{X_i(t)}{S(t)} = \frac{Z_i}{\zeta} ( \Lambda
(t) ), \qquad  \Lambda(t)= \frac{1}{4}\int_0^t S(u) \,du.
\]
Our first step is to eliminate the time change by studying the process
at a random stopping time $\varsigma_a$ when the process $S$ hits a
level $a$.

Consider the corresponding hitting time $\sigma_a$ for the process
$\zeta$. It then plainly follows from the time change relationship
$S(t)=\zeta(\Lambda(t))$ that $\Lambda(\varsigma_a)=\sigma_a$ and
%
%e19 ###
\begin{equation}\label{ridoftime}
\mu(\varsigma_a) = \frac{Z}{\zeta} ( \Lambda(\varsigma_a)
) = \frac{Z}{\zeta} ( \sigma_a )=\frac{1}{a}
Z ( \sigma_a ).
\end{equation}
On the right-hand side above, we have the process $Z$ the first time it
escapes from the set $a\usimp$. Since each $\theta_i$ is positive,
the BESQ process can only exit $\usimp$ through the oblique boundary
$\bndsimp$ (all the other boundaries are reflecting). Our objective is
to compute this exit density, which, in turn, gives the exit density of
the market weights at $\varsigma_a$.

Before proceeding to computations, we remark that it is enough to take
$a=1$. This is because of the following scaling property of BESQ
processes. Let $Y$ be a BESQ$^{\delta}_x$ process. Then, for any
positive $a$, the scaled process $ \{ {a^{-1}} Y(at), t \ge
0 \}$ is a BESQ$^\delta$\vadjust{\goodbreak} process starting from $x/a$. In
particular, by scaling each of $Z_1, \ldots, Z_n$ by $a$, we get that
the law of the vector $a^{-1}Z(\sigma_a)$ is the same as the vector
$Z(\sigma_1)$ when the initial vector of values has been divided by $a$.

The other consideration is whether or not $\sigma_a$ is finite. The
sum $\zeta$ is a BESQ process of dimension $\theta_0$. This process
is transient if and only if $\theta_0 > 2$. Thus, under the
assumptions in Proposition~\ref{exitderiv}, the finiteness of $\sigma
_1$ holds with probability one.

%s2.1 ###
\subsection{Green kernel and the exit density of BESQ processes}

Our main tool is the definition of a functional transformation
analogous to the classical \textit{Kelvin transform}. The intuition
comes from the fact that when the dimensions of BESQ processes are
positive integers, they have the same law as that of the Euclidean
norm-square of multidimensional Brownian motion. Thus, the exit density
from the unit simplex for BESQ processes can, in principle, be derived
from the Poisson kernel expansion for the exit density of the Brownian
motion from the unit ball. One way to obtain the Poisson kernel formula
is by employing classical Kelvin transform techniques (see the book on
harmonic function theory~\cite{harmonicfn}, Chapter 4). We generalize
that concept below.

Consider the (generalized) Markovian generator of the process $(Z_1,
Z_2, \ldots,\break Z_n)$ acting on $C^2 (\mathbb{R}^{n+} )$, the space
of functions that are twice continuously differentiable in $\mathbb{R}^{n+}$
up to the boundary. It is the following differential operator:
%
%e20 ###
\begin{equation}\label{besqgenl}
\mathcal{L} = \sum_{i=1}^n \theta_i \frac{\partial}{\partial x_i} +
2\sum_{i=1}^n z_i \frac{\partial^2}{\partial x^2_i}.
\end{equation}
Any twice continuously differentiable function $u$ that satisfies
$\mathcal{L} u=0$ in an (open) domain $D \subseteq\mathbb{R}^{n+}$
will be
called \textit{$\mathcal{L}$-harmonic} on $D$.

Define the inversion map $I\dvtx \mathbb{R}^{n+}\setminus\{0\}
\rightarrow\mathbb{R}
^{n+}$ by
%
%e21 ###
\begin{equation}\label{whatisI}
I (z ) = \frac{z}{ ( \sum_{i=1}^n z_i )^2}.
\end{equation}
It is easy to see that $I$ is one-to-one and $I\circ I$ is the identity
map. Also, $I$ inverts the interior of the punctured unit simplex
$\usimp\setminus\{0\}$ to the interior of
its complement in $\mathbb{R}^{n+}\setminus\{0\}$. If $D$ is a
domain in
$\mathbb{R}^{n+}\setminus\{0\}$, we will denote its image under the
inversion map by $I(D)$.

Let $D$ be a domain in $\mathbb{R}^{n+}\setminus\{0\}$ and let $u$
be a
real-valued function on~$D$. One can define a function $K[u]\dvtx  I(D)
\rightarrow\mathbb{R}$ as
%
%e22 ###
\begin{equation}\label{whatiskelvin}
K[u](z):= \Biggl( \sum_{i=1}^n z_i \Biggr)^{1-\theta_0/2} u (
I(z) ), \qquad  \theta_0 > 2.
\end{equation}
Thus, $K$ transforms a function on $D$ to a corresponding function on
$I(D)$. We prove that it takes $\mathcal{L}$-harmonic functions on $D$ to
$\mathcal{L}$-harmonic functions on $I(D)$. We have the following proposition.\vadjust{\goodbreak}

\begin{prop}\label{kelvintrans}
For any $C^2$ function $u$ on $D$, define
\[
\Psi(z)= \Biggl( \sum_{i=1}^n z_i \Biggr)^2 \mathcal{L} u(z), \qquad  z
\in D.
\]
$K[u]$ is then a $C^2$ function on $I(D)$, and we have
\[
\mathcal{L}K[u](z) = K [ \Psi](z) \qquad  \mbox{for all }
z\in I(D).
\]
Thus, if $u$ is $\mathcal{L}$-harmonic, then so is $K[u]$.
\end{prop}

To construct the proof we will need the following lemma.

\begin{lemma}\label{lemmahomog}
Let $p$ be a polynomial in $n$ variables that is homogeneous of degree
$m$. Then, on any domain $D \subseteq\mathbb{R}^{n+}\setminus\{0\}
$, we have
\[
\mathcal{L} \Biggl( \Biggl( \sum_{i=1}^n z_i \Biggr)^{1-\theta_0/2 -2m}
p(z) \Biggr)= \Biggl( \sum_{i=1}^n z_i \Biggr)^{1-\theta_0/2 - 2m}
\mathcal{L} p(z).
\]
\end{lemma}

\begin{pf}
First note that for any two $C^2$ functions $f,g$, we have
%
%e23 ###
\begin{equation}
\mathcal{L}(fg)=f\mathcal{L}(g) + g\mathcal{L}(f) + 4\sum_{i=1}^n z_i
\partial_i f \,\partial_i g.
\end{equation}

Now, for any power $r$ we have
\begin{eqnarray*}
\mathcal{L} \Biggl( \sum_{i=1}^n z_i \Biggr)^{r} &=& r\sum_{i=1}^n
\theta_i \Biggl( \sum_{j=1}^n z_j \Biggr)^{r-1} + 2r(r-1)\sum
_{i=1}^n z_i \Biggl( \sum_{j=1}^n z_j \Biggr)^{r-2}\\
&=& r\sum_{i=1}^n \theta_i \Biggl( \sum_{j=1}^n z_j \Biggr)^{r-1} +
2r(r-1) \Biggl( \sum_{j=1}^n z_j \Biggr)^{r-1}\\
&=& r (\theta_0 + 2r - 2 ) \Biggl( \sum_{i=1}^n z_j \Biggr)^{r-1}.
\end{eqnarray*}
Thus, using the product formula~\eqref{prodform} we get
\begin{eqnarray*}
\mathcal{L} \Biggl( \Biggl( \sum_{i=1}^n z_i \Biggr)^{r} p(z) \Biggr)&=&
\Biggl( \sum_{i=1}^n z_i \Biggr)^{r} \mathcal{L}(p) + r (\theta_0
+ 2r - 2 ) \Biggl( \sum_{i=1}^n z_i \Biggr)^{r-1} p(z)\\
&&{} + 4r\sum_{i=1}^n z_i \Biggl( \sum_{j=1}^n z_j
\Biggr)^{r-1}\partial_i p\\
&=& \Biggl( \sum_{i=1}^n z_i \Biggr)^{r} \mathcal{L}(p) + r (\theta
_0 + 2r - 2 ) \Biggl( \sum_{i=1}^n z_i \Biggr)^{r-1} p(z)\\
&&{} + 4rm
\Biggl( \sum_{i=1}^n z_i \Biggr)^{r-1}p(z).
\end{eqnarray*}
The final equality follows from the fact that for all homogeneous
polynomials of degree $m$, we should have $\langle z,\nabla p \rangle= mp$.
The easiest way to see this is to note that $p(\alpha z)= \alpha^m
p(z)$ for all $\alpha> 0$, take the derivative with respect to $\alpha
$ and finally put $\alpha=1$.

Thus, we get
\[
\mathcal{L} \Biggl( \Biggl( \sum_{i=1}^n z_i \Biggr)^{r} p(z)
\Biggr)= \Biggl( \sum_{i=1}^n z_i \Biggr)^{r} \mathcal{L}(p) + r ( \theta
_0 + 2r -2 + 4m ) \Biggl( \sum_{i=1}^n z_i \Biggr)^{r-1} p(z).
\]
Choosing $r$ such that $\theta_0 + 2r -2 + 4m =0$ proves the lemma.
\end{pf}

\begin{pf*}{Proof of Proposition~\ref{kelvintrans}}
We first prove this proposition when $u$ is a polynomial $p$,
homogeneous of degree $m$. By utilizing the property of homogeneity, we
can write
%
%e24 ###
\begin{eqnarray}\label{firsteq}
\mathcal{L} K[p] &=& \mathcal{L} \Biggl[ \Biggl( \sum_{i=1}^n z_i
\Biggr)^{1-\theta_0/2} p\Biggl ( \frac{z}{ (\sum_{i=1}^n z_i
)^2} \Biggr) \Biggr]\nonumber
\\
&=&\mathcal{L} \Biggl[ \Biggl( \sum_{i=1}^n z_i \Biggr)^{1-\theta_0/2 -
2m} p(z) \Biggr] \\
&=& \Biggl( \sum_{i=1}^n z_i \Biggr)^{1-\theta_0/2 -
2m}\mathcal{L}(p).
\nonumber
\end{eqnarray}
The final equality is due to Lemma~\ref{lemmahomog}.

Now, note that since $p$ is homogeneous of degree $m$, we have that
$\mathcal{L}(p)$ is homogeneous of degree $m-1$. Thus,
%
%e25 ###
\begin{eqnarray}\label{seceq}
K \Biggl[ \Biggl( \sum_{i=1}^n z_i \Biggr)^{2}\mathcal{L}(p) \Biggr] &=&
\Biggl( \sum_{i=1}^n z_i \Biggr)^{1-\theta_0/2 - 2} \mathcal{L}p
\biggl( \frac{z}{ ( \sum_i z_i )^2} \biggr)\nonumber
\\
&=& \Biggl( \sum_{i=1}^n z_i \Biggr)^{1-\theta_0/2 - 2 - 2(m-1)}
\mathcal{L} p(z)\\
&=& \Biggl( \sum_{i=1}^n z_i \Biggr)^{1-\theta_0/2 -
2m}\mathcal{L}(p).
\nonumber
\end{eqnarray}
Combining equalities~\eqref{firsteq} and~\eqref{seceq} we get
\[
\mathcal{L} K[p] = K\Biggl [ \Biggl( \sum_{i=1}^n z_i \Biggr)^{2}\mathcal{L}(p) \Biggr],
\]
which proves the proposition for the special case of homogeneous polynomials.

The general result now follows for all polynomials (obtained by taking
linear combinations of the homogeneous ones) and finally for all $C^2$
functions (by taking suitable limits of polynomial sequences).
\end{pf*}

The explicit description of a Kelvin transform allows us to compute the
Green function for the independent BESQ processes inside the unit
simplex. As before, consider $Z=(Z_1, \ldots, Z_n)$ to be a vector of
independent BESQ processes with respective dimensions $\theta_1,
\ldots, \theta_n$, satisfying the assumptions of Proposition \ref
{exitderiv}. In that case, the process $Z$ is transient (the sum $\zeta
$ being a BESQ process that is transient).

Let $p_t^{\theta}(x,y)$ denote the transition density of BESQ$^{\theta
}$. Define the potential kernel of $Z$ as follows:
\[
u_y(x)=u(x,y) = \int_0^{\infty} \prod_{i=1}^n p_t^{\theta_i} (
x_i, y_i )\,dt, \qquad  x, y \in\mathbb{R}^{n+}.
\]
We compute this kernel below.

\begin{prop}
The potential kernel $u(x,y)$, when $\theta_0=\sum_{i=1}^n \theta_i
> 2$, is given by the following formula:
%
%e26 ###
\begin{eqnarray}\label{whatisuxy}
u(x,y) &=&\frac{1}{2} S^{1-\theta_0/2}\prod_{i=1}^n y_i^{\theta
_i/2-1} \sum_{m=0}^{\infty} \Gamma( \theta_0/2 - 1 + 2m
) \frac{S^{-2m}}{m!} \nonumber
\\[-8pt]
\\[-8pt]
&&{}\times\sum_{k: k_1 + \cdots+ k_n =m} \combi{m}{k_1\cdots
k_n}\prod_{i=1}^n \frac{(x_iy_i)^{k_i}}{\Gamma(\theta_i/2 + k_i)},
\nonumber
\end{eqnarray}
where $S=\sum_{i=1}^n (x_i + y_i)$.
\end{prop}

\begin{pf}
The transition density of a BESQ process is explicitly described in
\cite{yorbook}, Appendix 7, page 549, to be $t^{-1}f(y/t, \theta,
x/t)$, where $f(\cdot,k,\lambda)$ is the density of a noncentral
chi-square distribution with $k$ degrees of freedom and a noncentrality
parameter value $\lambda$. In particular, it can be written as a
Poisson mixture of central chi-square (or gamma) densities. Thus, we
have the expansion
\[
p_t^{\theta}(x,y)= t^{-1}\sum_{k=0}^{\infty} e^{-x/2t}\frac
{(x/2t)^k}{k!} g_{\theta+2k}(y/t),
\]
where $g_{r}$ is the density of $\operatorname{Gamma}(r/2,1/2)$. Taking products over
$\theta_i$'s, we get
%
%e27 ###
\begin{equation}\label{bigprod}
\prod_{i=1}^n p_t^{\theta_i}(x_i,y_i) = t^{-n}\prod_{i=1}^n \Biggl[
\sum_{k_i=0}^{\infty} e^{-x_i/2t}\frac{(x_i/2t)^{k_i}}{k_i!}
g_{\theta_i+2k_i}(y_i/t) \Biggr].
\end{equation}

Recall the special notation introduced in Section~\ref{notations} to
keep track of the various product terms.

Since every term in~\eqref{bigprod} is nonnegative, we can expand the
product as a series and get
%
%e28 ###
\begin{eqnarray}\label{prodform}
t^n\prod_{i=1}^n p_t^{\theta_i}(x_i,y_i)
&=& \sum_{k_1, \ldots,k_n}
\prod_{i=1}^n e^{-x_i/2t} \frac{(x_i/2t)^{k_i}}{k_i!} g_{\theta
_i+2k_i}(y_i/t)\nonumber\\
&=& \sum_{k_1, \ldots,k_n} e^{-S_x/2t}\frac{x^k}{k!}(2t)^{-S_k} \prod
_{i=1}^n \frac{1}{\Gamma(\theta_i/2 + k_i)}2^{-\theta_i/2
-k_i}
\\
&&\hphantom{\sum_{k_1, \ldots,k_n} e^{-S_x/2t}\frac{x^k}{k!}(2t)^{-S_k} \prod
_{i=1}^n}
{}\times\biggl( \frac{y_i}{t} \biggr)^{\theta_i/2 - 1 + k_i}e^{-y_i/2t}\nonumber\\
&=& t^n \sum_{k_1, \ldots,k_n} e^{-(S_x+S_y)/2t} \frac
{x^k}{k!}(2t)^{-S_k}\beta({k}) (2t)^{-\theta_0/2 - S_k} y^{\nu+ k}.
\nonumber
\end{eqnarray}
Here, $\beta(k)$ denotes the constant given by $1/\beta(k)= \prod
_{i=1}^n {\Gamma(\theta_i/2 + k_i)}$.

To simplify~\eqref{prodform}, it will be convenient to define $S=S_x +
S_y= \sum_{i}(x_i + y_i)$. Thus, by regrouping terms we get
%
%e29 ###
\begin{equation}\label{bigprod2}
\prod_{i=1}^n p_t^{\theta_i}(x_i,y_i) =
y^{\nu}\sum_{m=0}^{\infty} e^{-S/2t} (2t)^{-\theta_0/2 - 2m} \sum
_{k_1 + \cdots+ k_n =m} \beta(k)\frac{x^k y^k}{k!}.
\end{equation}
For notational convenience let us define
%
%e30 ###
\begin{equation}\label{whatisbetac}
C(m)=\sum_{k_1 + \cdots+ k_n =m} \beta(k)\frac{x^k y^k}{k!}
\end{equation}
while we integrate out $t$ from the expression in~\eqref{bigprod2}.

Thus, we get
\[
u(x,y)= \int_0^{\infty} \prod_{i=1}^n p_t^{\theta_i}(x_i,y_i) \,dt =
y^{\nu}\sum_{m=0}^{\infty} C(m) \int_0^{\infty}e^{-S/2t}
(2t)^{-\theta_0/2 - 2m} \,dt.
\]

Evaluating the inner integral is easy. Changing the variable to
$w=1/2t$, we get
\begin{eqnarray*}
\int_0^{\infty}e^{-S/2t} (2t)^{-\theta_0/2 - 2m } \,dt
&=& \int
_0^{\infty} e^{-Sw}w^{\theta_0/2 + 2m }\frac{dw}{2w^2}\\
&=&
\frac{1}{2}\int_0^{\infty} w^{\theta_0/2 -2 + 2m}e^{-Sw} \,dw\\
 &=&
\frac{1}{2}\Gamma( \theta_0/2 - 1 + 2m )
S^{-2m+1-\theta_0/2}.
\end{eqnarray*}
Note that the assumption that $\theta_0 > 2$ is being used to show
that the integral above is finite when $m=0$.
This completes the derivation of the formula
\[
u(x,y)= \frac{1}{2} y^{\nu}\sum_{m=0}^{\infty} \Gamma(
\theta_0/2 - 1 + 2m ) S^{-2m+1-\theta_0/2} \sum_{k_i\ge0,
k_1 + \cdots+ k_n =m} \beta(k)\frac{x^k y^k}{k!}.
\]
The expression in~\eqref{whatisuxy} can be obtained from above by
dividing and multiplying by $m!$'s inside the infinite sum.
\end{pf}

\begin{lemma}\label{lemmauxy}
The potential kernel $u_y(x)$ is $\mathcal{L}$-harmonic in the
interior of $\mathbb{R}^{n+}$. Moreover, for a fixed value of $y$, it
has a
uniform decay of order $O(\sum_i x_i)^{1-\theta_0/2}$ as $\sum_i
x_i$ tends to infinity.
\end{lemma}

\begin{pf} The first claim follows from the fact that $u$ is the
potential kernel, although it can be verified through direct computation.

For the second claim, let $s_x=\sum_i x_i$ and $s_y=\sum_i y_i$. In
what follows, we assume that $s_x$ is much larger than $s_y$.

Since we fix $y$ and the vector $\theta$, it follows from \eqref
{whatisuxy} that there is a constant~$C$ (depending on $y$ and $\theta
$) such that
%
%e31 ###
\begin{eqnarray}\label{somebnd}
u(x,y)&\le &C s_x^{1-\theta_0/2} \sum_{m=0}^{\infty} (2m)! \frac
{s_x^{-2m}}{m!} \sum_{k: k_1 + \cdots+ k_n =m} \combi{m}{k_1\cdots
k_n}\prod_{i=1}^n \frac{(x_iy_i)^{k_i}}{k_i!}\nonumber
\\[-8pt]
\\[-8pt]
&=&C s_x^{1-\theta_0/2} \sum_{m=0}^{\infty} \frac{(2m) !}{(m!)^2}
s_x^{-2m} \sum_{k: k_1 + \cdots+ k_n =m} \combi{m}{k_1\cdots
k_n}^2\prod_{i=1}^n (x_iy_i)^{k_i}.
\nonumber
\end{eqnarray}

Recall the identity
\[
\sum_{k: k_1 + \cdots+ k_n =m} \combi{m}{k_1\cdots k_n}\prod
_{i=1}^n z_i^{k_i}= (z_1+ \cdots+ z_n)^m.
\]
Since all terms considered are positive, we get
\begin{eqnarray*}
&&\sum_{k: k_1 + \cdots+ k_n =m} \combi{m}{k_1\cdots k_n}^2\prod
_{i=1}^n (x_iy_i)^{k_i} \\
&& \qquad  \le\Biggl\{ \sum_{ k_1 + \cdots+ k_n =m} \combi{m}{k_1\cdots k_n}\prod
_{i=1}^n x_i^{k_i} \Biggr\}  \Biggl \{ \sum_{ k_1 + \cdots+ k_n =m}
\combi{m}{k_1\cdots k_n}\prod_{i=1}^n y_i^{k_i} \Biggr\}\\
&& \qquad = s_x^m s_y^m.
\end{eqnarray*}

Substituting this bound into~\eqref{somebnd}, we get
\[
u(x,y)\le C s_x^{1-\theta_0/2} \sum_{m=0}^{\infty} \combi{2m}{m}
\biggl(\frac{s_y}{s_x} \biggr)^m \le C s_x^{1-\theta_0/2},
\]
when $s_x$ is large enough. This proves the claim.
\end{pf}

For the next proposition recall the stopping time $\sigma_1$, which is
the hitting time of level one for the sum process $\zeta$.

\begin{prop}\label{propgreen}
Consider the functional transform defined in~\eqref{whatiskelvin} and
define the kernel
%
%e32 ###
\begin{equation}\label{greensimp}
v(x,y) = u_y(x) - K[u_y](x)= u(x,y) - K[u_y](x), \qquad  x, y \in\mathbb{R}
^{n+}.
\end{equation}
Then, for every smooth nonnegative function $f$ which is compactly
supported away from the origin and any $x\in\usimp$, we get
%
%e33 ###
\begin{equation}\label{greeneq}
E_x\int_0^{\sigma_1} f(Z(s))\,ds = \int_{\usimp} f(y) v(x,y) \,dy.
\end{equation}
In other words, $v$ is the Green potential on the unit simplex $\usimp
$ for the process~$Z$.

Explicitly, the kernel $v(x,y)$ is equal to
%
%e34 ###
\begin{eqnarray}\label{whatisv}
&&   \frac{y^{\nu}}{2}\sum_{m=0}^{\infty} \frac{\Gamma(\theta_0/2 -
1 + 2m)}{m!} \{ ( S_x + S_y )^{-2m+1-\theta_0/2} -
( S_xS_y + 1 )^{-2m+1-\theta_0/2} \}\nonumber\hspace*{-35pt}
\\[-8pt]
\\[-8pt]
&& \hphantom{\frac{y^{\nu}}{2}\sum_{m=0}^{\infty}}
 {}\times\sum_{k_i\ge0: k_1 + \cdots+ k_n=m} \combi{m}{k_1\cdots
k_n} \prod_{i=1}^{n} \frac{(x_iy_i)^{k_i}}{\Gamma(\theta_i/2 + k_i)}.
\nonumber\hspace*{-35pt}
\end{eqnarray}
Thus, $v$ satisfies the symmetry property
%
%e35 ###
\begin{equation}\label{vsymm}
v(x,y)\prod_{i=1}^n x_i^{\theta_i/2 -1} = v(y,x) \prod_{i=1}^n
y_i^{\theta_i/2 -1}.
\end{equation}
\end{prop}

\begin{pf} To prove this proposition we first note that for any
compactly supported (in $\mathbb{R}^{n+}$) smooth test function $f$,
we have
\[
U(f)(x):=\int_{\mathbb{R}^{n+}} f(y) u(x,y)\,dy = E_x \int_0^{\infty} f
( Z(s) ) \,ds.
\]
Thus, by the Markov property, it follows that $M_1(t)=U(f)(Z(t)) +\break \int
_0^t f(Z(s))\,ds$ is a martingale [$Z(0)=x$] and that (Lemma \ref
{lemmauxy}) $\mathcal{L}u_y(x)=0$ for all $x \notin\usimp$ when
$y\in
\usimp$.

Fix a $y \in\usimp\setminus\bndsimp$. We now use Proposition \ref
{kelvintrans} for the domain $D=\{ x \in\mathbb{R}^{n+}\dvtx  x_i > 0, x
\notin\usimp\}$. It then follows that $K[u_y](x)$ is $\mathcal
{L}$-harmonic for all $x$ in $I(D)$, which is the interior of $\usimp$.

\begin{claim*} We now claim that if we define
\[
W(f)(x) := \int_{\mathbb{R}^{n+}} f(y) K[u_y](x) \,dy,
\]
then $M_2(t)= W(f)(Z(t\wedge\sigma_1 \wedge\sigma_{\epsilon}))$ is
a martingale for every $\epsilon> 0$ when \mbox{$Z(0) \in\usimp$}.
\end{claim*}

To prove this claim, it suffices to show that $K[u_y](Z(t\wedge\sigma
_1 \wedge\sigma_{\epsilon}))$ is a martingale for every $y$ in
$\usimp$. We apply It\^o's rule to the function $K[u_y]$. Since
$K[u_y](x)$ is $\mathcal{L}$-harmonic in the interior of $\usimp$, the
process $K[u_y](Z(t\wedge\sigma_1 \wedge\sigma_{\epsilon}))$ is a
local martingale with the decomposition
\[
dK[u_y]\bigl(Z(t\wedge\sigma_1 \wedge\sigma_{\epsilon})\bigr) = 2\sum
_{i=1}^n \sqrt{Z_i(t\wedge\sigma_1 \wedge\sigma_{\epsilon})}
\frac{\partial}{\partial z_i} K[u_y](Z(t))\,d\beta_i(t\wedge\sigma
_1 \wedge\sigma_{\epsilon}).
\]
The square bracket of this local martingale can be easily computed as
\begin{eqnarray*}
&&d \langle K[u_y](Z) \rangle(t\wedge\sigma_1 \wedge\sigma_{\epsilon})\\
&& \qquad =
4\sum_{i=1}^n Z_i(t\wedge\sigma_1 \wedge\sigma_{\epsilon})
\biggl\{ \frac{\partial}{\partial z_i} K[u_y]\bigl(Z(t\wedge\sigma_1 \wedge
\sigma_{\epsilon})\bigr) \biggr\}^2\,dt.
\end{eqnarray*}

Let us now compute the partial derivative:
\begin{eqnarray*}
\frac{\partial}{\partial x_i} K[u_y](x) &=& \frac{\partial}{\partial
x_i} \Biggl[ \biggl( \sum_i x_i \biggr)^{1-\theta_0/2} u \biggl(
\frac{x}{(\sum_i x_i)^2} \biggr) \Biggr]\nonumber\\
&=& (1-\theta_0/2) \biggl( \sum_i x_i \biggr)^{-\theta_0/2} u \biggl(
\frac{x}{(\sum_i x_i)^2} \biggr) \\
&&{}+ \biggl( \sum_i x_i \biggr)^{1-\theta_0/2}\sum_{j=1}^n u_j \biggl(
\frac{x}{(\sum_i x_i)^2} \biggr) \biggl[ \frac{1\{i=j\}}{(\sum_i
x_i)^2} - \frac{2x_j}{(\sum_i x_i)^3} \biggr].\nonumber
\end{eqnarray*}
Here, $u_j$ denotes the $j$th partial derivative of $u_y$.

Now, it can be seen from its explicit series expansion \eqref
{whatisuxy} that $u_y$ is bounded and has bounded partial derivatives
in $\usimp^c$ when $\sum_i y_i < 1$. Thus, from the expression above,
the partial derivatives of $K[u_y](x)$ are uniformly bounded when $x\in
\usimp$ and $\sum_i x_i > \epsilon> 0$. Hence, it follows that
$K[u_y](Z)(t \wedge\sigma_1 \wedge\sigma_{\epsilon})$ is a
martingale. By integrating with respect to $f(y)\,dy$, we have shown that
$M_2(t\wedge\sigma_{\epsilon})$ is a martingale for every $\epsilon
> 0$.

Thus, the process
\[
N(t)=U(f)\bigl(Z(t\wedge\sigma_1 \wedge\sigma_{\epsilon})\bigr) -
W(f)\bigl(Z(t\wedge\sigma_1 \wedge\sigma_{\epsilon})\bigr) + \int
_0^{t\wedge\sigma_1 \wedge\sigma_{\epsilon}} f ( Z(s)
) \,ds
\]
is also a martingale.

We now apply the optional sampling theorem to this martingale at the
stopping time $\sigma_1\wedge\sigma_{\epsilon}$ [notice that the
martingale is bounded\vadjust{\goodbreak} by $c_1 + c_2 (\sigma_1 \wedge\sigma_\epsilon
),$ which has a finite expectation]. There are two cases to consider.
When $\sigma_1 < \sigma_{\epsilon}$, we have $Z(t\wedge\sigma
_1\wedge\sigma_{\epsilon}) \in\bndsimp$. For any $x \in\bndsimp
$, we have $u_y(x)= K[u_y](x)$ and hence
\[
U(f)\bigl(Z( \sigma_1 \wedge\sigma_{\epsilon})\bigr) - W(f)\bigl(Z( \sigma_1
\wedge\sigma_{\epsilon})\bigr)=0  \qquad \mbox{when } \sigma_1 <
\sigma_{\epsilon}.
\]
In the other case (when $\sigma_{\epsilon} < \sigma_1$), we get
$\zeta(\sigma_1 \wedge\sigma_{\epsilon})=\epsilon$. Thus,
%
%e36 ###
\begin{equation}\label{wfromu}
W(f)\bigl(Z( \sigma_1 \wedge\sigma_{\epsilon})\bigr)= \epsilon^{1-\theta
_0/2}U(f)\bigl(\epsilon^{-2}Z( \sigma_1 \wedge\sigma_{\epsilon})\bigr).
\end{equation}

Thus, we get
\begin{eqnarray*}
&&E_x  \int_0^{\sigma_1 \wedge\sigma_{\epsilon}} f ( Z(s)
) \,ds\\
&& \qquad  = \int_{\usimp}f(y)v(x,y)\,dy\\
&& \quad  \qquad {}+ E_x [ -U(f)(Z( \sigma_{\epsilon})) + W(f)(Z( \sigma
_{\epsilon})) \mid\sigma_{\epsilon} < \sigma_1 ] P (
\sigma_{\epsilon} < \sigma_1 ).
\end{eqnarray*}
Since $\zeta$ has dimension greater than two, it almost surely does
not hit the origin. Thus, it is clear that as $\epsilon$ tends to
zero, the left-hand side of the above equation converges to $E_x\int
_0^{\sigma_1} f ( Z(s) ) \,ds$. We now show that the
right-hand side converges to $\int_{\usimp} f(y)v(x,y)\,dy$.

Using the scale functions for $\zeta$, it is easy to see that
%
%e37 ###
\begin{equation}\label{scalefn}
P ( \sigma_{\epsilon} < \sigma_1 )=\frac{S_x^{1-\theta
_0/2}-1}{\epsilon^{1-\theta_0/2} - 1}= O ( \epsilon^{\theta
_0/2 -1} ).
\end{equation}
Now, as $\epsilon$ tends to zero, $U(f)(Z( \sigma_{\epsilon}))$
remains bounded. The easiest way to see this is to note that $f$ has
compact support away from the origin, and $Z( \sigma_{\epsilon})$ is
away from all points in the support for sufficiently small $\epsilon$. Hence,
\[
\lim_{\epsilon\rightarrow0} E_x [ -U(f)(Z( \sigma_{\epsilon
}))\mid\sigma_{\epsilon} < \sigma_1 ] P ( \sigma
_{\epsilon} < \sigma_1 )=0.
\]

On the other hand, from~\eqref{wfromu} and~\eqref{scalefn} we get
\begin{eqnarray*}
&&\lim_{\epsilon\rightarrow0} E_x [ W(f)(Z( \sigma_{\epsilon
}))\mid\sigma_{\epsilon} < \sigma_1 ] P ( \sigma
_{\epsilon} < \sigma_1 )\\
&& \qquad = S_x^{1-\theta_0/2}\lim_{\epsilon
\rightarrow0} E_x [ U(f)(\epsilon^{-2}Z(\sigma_{\epsilon
}))\mid\sigma_{\epsilon} < \sigma_1 ].
\end{eqnarray*}
The limit on the right is zero since $\epsilon^{-2}Z(\sigma_{\epsilon
})$ tends to infinity, and thus the function $U(f)$ applied to it
uniformly goes to zero, by Lemma~\ref{lemmauxy}.

This completes the proof of the equality in~\eqref{greeneq}.

Now, we compute the kernel $v(x,y)$ from the formula \eqref
{whatisuxy}. We introduce the temporary notation
\[
R_m = \frac{\Gamma(\theta_0/2 - 1 + 2m)}{m!}\sum_{k: k_1 + \cdots
+ k_n=m} \combi{m}{k_1\cdots k_n} \prod_{i=1}^{n} \frac
{(x_iy_i)^{k_i}}{ \Gamma(\theta_i/2 + k_i)}.
\]
Thus, from~\eqref{whatisuxy}, we get
\begin{eqnarray*}
v(x,y)
&=& u(x,y) - \Biggl( \sum_{i=1}^n x_i \Biggr)^{1-\theta
_0/2}u \biggl( \frac{x}{(\sum_i x_i)^2} ,y \biggr)\\
&=& \frac{y^{\nu}}{2}S^{1-\theta_0/2}\sum_{m=0}^{\infty} R_m
S^{-2m}\\
&&{}-S_x^{1-\theta_0/2}\frac{y^{\nu}}{2} \biggl( \frac{1}{ S_x} + S_y
\biggr)^{1-\theta_0/2} \sum_{m=0}^{\infty} R_m S_x^{-2m} \biggl(
\frac{1}{ \sum_i x_i} + \sum_i y_i \biggr)^{-2m}\\
&=& \frac{y^{\nu}}{2}S^{1-\theta_0/2}\sum_{m=0}^{\infty} R_m
S^{-2m}\\
&&{}-\frac{y^{\nu}}{2} ( 1 + S_xS_y )^{1-\theta_0/2} \sum
_{m=0}^{\infty} R_m ( 1+ S_x S_y )^{-2m}\\
&=& \frac{y^{\nu}}{2}\sum_{m=0}^{\infty} R_m \{ ( S_x +
S_y )^{-2m+1-\theta_0/2} - ( S_xS_y + 1
)^{-2m+1-\theta_0/2} \}.
\end{eqnarray*}
This completes the derivation of the Green kernel.
\end{pf}

Finally, we derive the exit distribution from $\usimp$ for the process
$(Z_1, Z_2, \ldots,\break Z_n)$. Note that the transition density can be
guessed from the following version of Green's second identity for the
generator of the BESQ processes.
\begin{lemma}\label{besqgreen}
Let $\mathcal{L}$ be the generator in~\eqref{besqgenl}. Let $D
\subseteq
\usimp$ be a compact domain with piecewise smooth boundary. Let
$m=(m_1, \ldots, m_n)$ be the vector given by $m_i=x_i n_i(x)$, where
$n$ is the outward unit normal vector at a boundary point $x$.

Let $\omega(x)$ be the weight function $\omega(x) := x^{\nu}= \prod
_{j=1}^n x_j^{\theta_j/2 -1}$
and let $u,v$ be two functions on $D$ which are twice continuously
differentiable (continuous up to the boundary) on $D$.

Then, assuming that the right-hand side below is integrable, we have
%
%e38 ###
\begin{equation}\label{green2eq}
\int_{D} ( u\mathcal{L}v - v\mathcal{L}u )\omega(x) \,dx = 2
\int_{\partial D} \biggl(u \frac{\partial v}{\partial m} - v \frac
{\partial u}{\partial m} \biggr) \omega(x)\sigma(dx).
\end{equation}
Here, $\sigma$ is the surface Lebesgue measure on $\partial D$.
\end{lemma}

\begin{pf} Recall the generator $\mathcal{L}$ for BESQ processes:
\[
\mathcal{L}v = \sum_{i=1}^n \theta_i \frac{\partial}{\partial
x_i} v
+ 2\sum_{i=1}^n x_i \frac{\partial^2}{\partial x^2_i} v= 2\sum
_{i=1}^n x_i^{1-\theta_i/2}\frac{\partial}{\partial x_i}\biggl [
x_i^{\theta_i/2} \frac{\partial v}{\partial x_i} \biggr].
\]

We now use the divergence theorem. Let $n$ denote the outward unit
normal vector on a compact domain $D$ with a piecewise smooth boundary.
Given a vector field of continuously differentiable functions $F=(f_1,
\ldots, f_n)$ on $D$, define $\operatorname{div} F = \sum_{i=1}^n
\partial
_i f_i$. Then,
\[
\int_D \operatorname{div} F(x) \,dx = \int_{\partial D} F\cdot n(y)
\sigma(dy).
\]
Here, $\partial D$ is the boundary of $D$, and $\sigma(dy)$ is the
surface Lebesgue measure on $\partial D$.

Let $\mu$ be the measure on $\usimp$ given by the density $\omega
(x)$. We now use the divergence theorem to derive the following
multivariate integration by parts:
\begin{eqnarray*}
\int_{D} u \mathcal{L} v\, d\mu
&=& 2 \int_{D} \sum_{i=1}^n u
x_i^{1-\theta_i/2} \frac{\partial}{\partial x_i} \biggl[ x_i^{\theta
_i/2} \frac{\partial v}{\partial x_i} \biggr] \prod_{j=1}^n
x_j^{\theta_j/2 -1} \,dx\\
&=& - 2 \int_{D} \sum_{i=1}^n x_i \frac{\partial u}{\partial x_i}
\frac{\partial v}{\partial x_i} \prod_{j=1}^n x_j^{\theta_j/2 -1} \,dx
+ \int_{\partial D} F\cdot n(y) \sigma(dy).
\end{eqnarray*}
The last equality above is obtained by applying the divergence theorem
to the function $F=(f_1, \ldots, f_n),$ where
\[
f_i = 2 x_i u \frac{\partial v}{\partial x_i} \prod_{j=1}^n
x_j^{\theta_j/2 -1}=2 u \biggl[ x_i^{\theta_i/2} \frac{\partial
v}{\partial x_i} \biggr]\prod_{j\neq i} x_j^{\theta_j/2 -1}.
\]
Interchanging $u$ and $v$ above and taking a difference, we arrive at
\eqref{green2eq}.
\end{pf}

\begin{pf*}{Proof of Proposition~\ref{exitderiv}}
We first use the previous lemma to produce a convincing \textit
{heuristic} derivation of the transition density.

\textit{Step \textup{1:} Heuristics.} It suffices to prove formula
\eqref{exitdensity} for an arbitrary $z$ in the open unit simplex. Fix
any $\epsilon> 0$, small enough that $B(z, \epsilon)$ is contained in
the interior of $\usimp$.

Consider a smooth nonnegative function $f$ on $\bndsimp$. Let $\psi$
be a function on $\mathbb{R}^{n+}$ which is nonnegative, smooth and zero
outside $B(z, \epsilon)$. We will use $\psi$ as an approximation of
the delta mass at $z$.

Consider the two functions defined on $\usimp$: $h(x)=\int v(x,y)\psi
(y) \,dy$ and $l(x)= E_x f ( Z_{\sigma_1} )$. In the
interior of $\usimp$ we have $\mathcal{L} h =-\psi$, by virtue of $v$
being the Green potential, and $l(x)$ is $\mathcal{L}$-harmonic, as a
corollary of its definition. Assuming that both of these functions are
also smooth, with derivatives extending continuously to the boundary,
we can apply the extended Green's identity~\eqref{green2eq} for $u=l$,
$v=h$ and $D=\usimp$ to get
%
%e39 ###
\begin{equation}\label{ibp1}
-\int_{B(z,\epsilon)} l(y) \psi(y) \omega(y) \,dy = 2 \int_{\partial
\usimp} \biggl( l \frac{\partial h}{\partial m} -h \frac{\partial
l}{\partial m} \biggr) \omega(x)\sigma(dx).
\end{equation}
Let us now analyze the right-hand side of the above equation. The
surface $\partial\usimp$ is piecewise linear and consists of the
subsets $S_1, S_2, \ldots, S_n$ and $S_0=\bndsimp$, where the outward
normal vector for $S_i$ is $-e_i$ for $i=1,2,\ldots,n$, and for
$\bndsimp$, the vector is $n^{-1/2}1$, the normalized vector of all
1's. Thus, integrating separately on each $S_i$ and temporarily
dropping the constant $2$, we get
%
%e40 ###
\begin{eqnarray}\label{ibp2}
\int_{\partial\usimp} \biggl( l \frac{\partial h}{\partial m} -h
\frac{\partial l}{\partial m} \biggr) \omega(x)\sigma(dx)&=& \sum
_{i=0}^n \int_{S_i} \biggl( l \frac{\partial h}{\partial m} - h \frac
{\partial l}{\partial m} \biggr) \omega(x)\sigma(dx)\nonumber
\\
&=&- \sum_{i=1}^n \int_{S_i} x_i \biggl( l \frac{\partial h}{\partial
x_i} - h \frac{\partial l}{\partial x_i} \biggr) \omega(x)\sigma
(dx)\\
&&{}+ \frac{1}{\sqrt{n}}\int_{\bndsimp} \sum_{i=1}^n x_i \biggl( l
\frac{\partial h}{\partial x_i} - h \frac{\partial l}{\partial x_i}
\biggr) \omega(x)\sigma(dx).
\nonumber
\end{eqnarray}
Over each $S_i$, for $i=1,2,\ldots,n$, the $i$th coordinate $x_i$ is
zero. Due to the fact that each $\theta_i > 0,$ and assuming that $h$,
$l$ and their partial derivatives are well behaved, the integral above
must be zero. Over $\bndsimp$, by definition, we have $h=0$ and $l=f$.
Thus, combining~\eqref{ibp1} and~\eqref{ibp2} we get
%
%e41 ###
\begin{equation} \quad
-\int_{B(z,\epsilon)} l(y) \psi(y) \omega(y) \,dy = 2n^{-1/2} \int
_{\bndsimp} \sum_{i=1}^n x_i \biggl( f(x) \frac{\partial h}{\partial
x_i} \biggr) \omega(x) \sigma(dx).
\end{equation}
We now take a sequence of $\psi$'s, functions approximating the delta
function, such that both $h$ and its partial derivatives converge to
$v_z$ and its corresponding partial derivatives. We thus infer
that
%
%e42 ###
\begin{equation}\label{ibp3}
l(z) \omega(z)=-2\int_{\bndsimp} f(x) \omega(x) \sum_{i=1}^n x_i
\frac{\partial}{\partial x_i} v(x,z) \,dx.
\end{equation}
The $1/\sqrt{n}$ factor gets absorbed when we parametrize the surface
$\bndsimp$ by $\mathbb{R}^{n-1}$ (and hence $dx$ represents the induced
measure from the Lebesgue measure on $\mathbb{R}^{n-1}$).

Since $l(z)=E_{z} f(Z_{\sigma_1})$ this identifies the exit density as
%
%e43 ###
\begin{equation}\label{whatisphiv}
\varphi_z(x)=\varphi(z,x)=-2\frac{\omega(x)}{\omega(z)} \Biggl[\sum
_{i=1}^n x_i \frac{\partial}{\partial x_i}v(x,z) \Biggr], \qquad  x\in
\bndsimp, z\in\usimp.
\end{equation}
The problem with the above argument is that a priori we do not know the
regularity of the exit distribution at the boundary of the simplex.
However, once we have guessed the solution, we can easily check that it
must be the correct one.

\textit{Step \textup{2:} Computation based on heuristics.} Let us now
compute explicitly the expression~\eqref{whatisphiv}. To simplify
matters, we introduce some temporary notation: for $m=0,1,2,\ldots$ let
\begin{eqnarray*}
B_m &=& \Biggl( \sum_{i=1}^n z_i + \sum_{i=1}^n x_i
\Biggr)^{-2m+1-\theta_0/2} - \Biggl( \Biggl(\sum_{i=1}^n z_i \Biggr)\Biggl (
\sum_{i=1}^n x_i \Biggr) + 1 \Biggr)^{-2m+1-\theta_0/2},\\
D_m &=& \sum_{k: k_1 + \cdots+ k_n=m} \combi{m}{k_1\cdots k_n} \prod
_{i=1}^n \frac{(z_i x_i)^{k_i}}{\Gamma(\theta_i/2 + k_i)}.
\end{eqnarray*}
Thus, from~\eqref{greensimp}, we get
\[
\frac{\partial}{\partial x_i} v(x,z) = \frac{z^{\nu}}{2} \sum
_{m=0}^{\infty} \frac{\Gamma(\theta_0/2 - 1 + 2m)}{m !} \biggl[ D_m
\frac{\partial}{\partial x_i} B_m + B_m \frac{\partial}{\partial
x_i} D_m \biggr].
\]
Now, when $z$ is in the open unit simplex and $x\in\bndsimp$, we get
\begin{eqnarray*}
\frac{\partial}{\partial x_i}  B_m &=& (-2m+1-\theta_0/2) \Biggl( \sum
_{i=1}^n z_i + 1 \Biggr)^{-2m-\theta_0/2}\\
&&{}- (-2m+1-\theta_0/2) \Biggl(\sum_{i=1}^n z_i \Biggr)\Biggl ( \sum
_{i=1}^n z_i + 1 \Biggr)^{-2m-\theta_0/2}\\
&=&(1-\theta_0/2-2m) \Biggl(1-\sum_{i=1}^n z_i \Biggr)\Biggl ( \sum
_{i=1}^n z_i + 1 \Biggr)^{-2m-\theta_0/2}.
\end{eqnarray*}
We do not need to compute partial derivatives of $D_m$ since $B_m$ is
zero on~$\bndsimp$.

Thus, by combining the partial derivatives of $B_m$, we get that
$\varphi(z,x)$ is equal to
%
%e44 ###
\begin{eqnarray}\label{exitdensity0}
&&\omega(x)\sum_{m=0}^{\infty}\frac{\Gamma(\theta_0/2 - 1 +
2m)}{m!}( 2m + \theta_0/2 -1)\nonumber
\\[-8pt]
\\[-8pt]
&&\hphantom{\omega(x)\sum_{m=0}^{\infty}}{}\times\Biggl(1-\sum_{i=1}^n z_i \Biggr)
\Biggl(1+ \sum_{i=1}^n z_i \Biggr)^{-2m-\theta_0/2}D_m,
\nonumber
\end{eqnarray}
which leads to formula~\eqref{exitdensity} by substituting $S_z$ for
$\sum_i z_i$, noting that $\Gamma(\theta_0/2 - 1 + 2m)( 2m + \theta
_0/2 -1)=\Gamma(\theta_0/2+2m)$ and completing the terms in the
$\operatorname{Dir}$ density.

\textit{Step \textup{3:} A rigorous proof.}
We now show rigorously that the above formula is the true exit density.
To do this we merely need to check that the heuristic derivation shown
in step~1 goes through.

We first claim that $\varphi(z,x)$, as given by~\eqref{whatisphiv},
is in the kernel of $\mathcal{L}$ in the first coordinate. That is,
$\mathcal{L}_z \varphi(z,x)=0$ for all $z$ in the open unit simplex when\vadjust{\goodbreak}
$x\in\bndsimp$. To see this, we use the symmetry property of the
Green potential~\eqref{vsymm}. Thus,
%
%e45 ###
\begin{eqnarray}\label{expression1}
\varphi(z,x)&=&-2\frac{\omega(x)}{\omega(z)} \Biggl[\sum_{i=1}^n x_i
\frac{\partial}{\partial x_i}v(x,z) \Biggr]= -2\omega(x) \sum
_{i=1}^n x_i \frac{\partial}{\partial x_i} \frac{v(x,z)}{\omega
(z)}\nonumber
\\[-8pt]
\\[-8pt]
&=& -2\omega(x) \sum_{i=1}^n x_i \frac{\partial}{\partial x_i} \frac
{v(z,x)}{\omega(x)}.
\nonumber
\end{eqnarray}
Since $\mathcal{L}_zv(z,x)=0$ for all $z$ in the interior of $\usimp$,
it immediately follows that $\mathcal{L}_z\varphi(z,x)$ must also be zero.

Now, consider any smooth test function $f$ on $\bndsimp$ and,
following step 1, for any $y\neq0$ in $\usimp\setminus\bndsimp$, define
%
%e46 ###
\begin{equation}\label{newl}
l(y)= \int_{\bndsimp} f(x) \varphi(y,x) \,dx.
\end{equation}
Now, from the explicit formula for $\varphi(\cdot,x)$ in \eqref
{exitdensity0} and usual analysis, it follows that $\partial\varphi
/\partial m$ (and hence $\partial l/ \partial m$) is a well-defined
power series up to the boundary of $\usimp$. The function $h$ in
\eqref{ibp1} is a convolution with a smooth mollifier $\psi$ and is
obviously smooth. Thus,~\eqref{ibp1} and~\eqref{ibp2} go through and
we arrive at the following modification of~\eqref{ibp3}:
%
%e47 ###
\begin{equation}\label{ibp4}
  l(z) \omega(z) = -2 \int_{\bndsimp} l(x) \omega(x) \sum_{i=1}^n
x_i \frac{\partial}{\partial x_i} v(x,z)\,dx= \int_{\bndsimp}
l(x)\omega(z) \varphi(z,x) \,dx.\hspace*{-35pt}
\end{equation}

We now cancel the factor $\omega(z)$ appearing on both sides of \eqref
{ibp4} and compare with the definition of $l(z)$ in~\eqref{newl} to get
%
%e48 ###
\begin{equation}\label{ibp10}
\int_{\bndsimp} f(x) \varphi(z,x) \,dx = \int_{\bndsimp} l(x)
\varphi(z,x) \,dx.
\end{equation}
We now vary $z$ in $\usimp$ and use the following claim.

\begin{claim*} If $u\dvtx  \bndsimp\rightarrow\mathbb{R}$ is a
bounded continuous function such that
%
%e49 ###
\begin{equation}\label{claimcond}
\int_{\bndsimp} u(x) \varphi(z,x) \,dx =0  \qquad \mbox{for all } z\in
\usimp,
\end{equation}
then $u$ is identically zero.
\end{claim*}

Assuming the claim for now, we get $l(x)=f(x)$ for all $x \in\bndsimp
$. Thus, the function $l$ is harmonic with the correct boundary
condition. The rest of the proof is now immediate. Since our density
$\varphi$ is supported over a compact set, it is enough to evaluate
expectations of monomials under the density. For any monomial $p=\prod
_{i=1}^n x_i^{\gamma_i}$ with each $\gamma_i \ge1$, we consider the function
\[
H(z) = \int_{\bndsimp} p(x) \varphi(z,x)\,dx, \qquad  z\in\usimp.
\]
We claim that the process $H(Z)(t\wedge\sigma_1 \wedge\sigma
_{\epsilon})$ is a martingale for any $\epsilon>0$. We have already
shown that $\mathcal{L}H=0$ inside $\usimp$ (since $\mathcal{L}
\varphi
=0,$ and taking the derivative inside the integral sign, which is
allowed since everything is smooth and bounded). Thus, the claim
follows by noting from the explicit expansion of $\varphi$ that the
first partial derivatives of $H$ are bounded on any domain away from
the origin. By applying the optional sampling theorem we get $H(Z(0))=
E H(Z(\sigma_1 \wedge\sigma_{\epsilon}))$. We can now take
$\epsilon$ to zero, arguing exactly as in the proof of Proposition
\ref{propgreen}, to claim that $H(Z(0))= E H(Z(\sigma_1))$. However,
by~\eqref{ibp10} and the above claim, we see that $H(Z(\sigma
_1))=p(Z(\sigma_1))$. Hence, we get
\[
E_z p ( Z(\sigma_1) )= \int_{\bndsimp} p(x) \varphi(z,x)\,dx.
\]
Since the above identity holds for all monomials $p$, this completes
the proof.

It remains to prove the above claim. Note that since $\varphi$ is a
power series, one can take the integral inside the sum in expression
\eqref{exitdensity}. Let $\varphi^u$ be the resulting power series in
$z$ which is identically zero under~\eqref{claimcond}. Let us define a
change of variables, $q_i=(1+ S_z)^{-2}z_i,$ in expression \eqref
{exitdensity} and expand $\varphi^u$ as below:
%
%e50 ###
\begin{equation}\label{varphiu}\qquad
\varphi^u(z)= (1-S_z)(1+S_z)^{-\theta_0/2}\sum_{m=0}^\infty t(m)
\sum_{k_1+ \cdots+ k_n=m} \rho(\mathbf{k}) \prod_{i=1}^n q_i^{k_i}
\equiv0.
\end{equation}
Here, $t(m)$ is the coefficient of the $m$th internal summand, as in
\eqref{exitdensity}, and
\[
\rho(\mathbf{k})= \combi{m}{k_1\cdots k_n} \int_{\bndsimp} u(x)
\operatorname{Dir}(x, k+ \theta/2) \,dx.
\]

Since the right-hand side of~\eqref{varphiu} is identically zero for
all $z\in\usimp$, it follows that the inner power series (as a
function of the variables $q_1, \ldots, q_n$) must be zero.

Now, fix any collection of positive integers $k_1, \ldots, k_n$. By
taking repeated partial derivatives $\partial_1^{k_1}\cdots\partial
_n^{k_n}\varphi^u(q)$ and letting each $q_i$ tend to zero, we obtain
that $\rho(\mathbf{k})$ must be zero for all $\mathbf{k}=(k_1,
\ldots, k_n)$. However, from the structure of the Dirichlet densities,
this shows that all multivariate moments of the measure $u(x)\omega
(x)\,dx$ on $\bndsimp$ must be zero. Since $\bndsimp$ is compact, this
identifies $u$ as the zero function.
\end{pf*}

The proof of Proposition~\ref{vsmmarket} follows immediately by
combining Proposition~\ref{exitderiv}, equation~\eqref{ridoftime} and
the discussion following it. One simply needs to keep in mind that in
keeping with the time change relationship~\eqref{vsmbes}, to compute
the distribution of market weights under the model $V(\delta_1, \ldots
, \delta_n),$ we need to compute the exit density for $n$ independent
BESQ's with dimensions $\theta_i=2\delta_i$ for $i=1,2,\ldots,n$.

Readers might wonder if there is any direct way of seeing that the
density expression in~\eqref{exitdensity} [and hence~\eqref{mwtden}]
integrates to 1. Although not elementary, the following argument is
such a direct method and serves as a sanity check.

Assume, for notational simplicity, that $\theta_0/2$ is a positive
integer $r$. Let $s$ denote $S_z$. Since the integral of each
$\operatorname{Dir}$
density appearing on the right-hand side of~\eqref{exitdensity} equals
1, by an application of the Fubini--Tonelli theorem, we get
%
%e51 ###
\begin{eqnarray}\label{direct}
\int_{\bndsimp} \varphi_z(y)\,dy&=& (1-s) \sum_{m=0}^\infty\combi
{2m+r -1}{m} (1+s)^{-2m-r}\nonumber\\
&&\hphantom{(1-s) \sum_{m=0}^\infty}
{}\times \sum_{k_1+ \cdots+ k_n=m} \combi{m}{k_1
\cdots k_n} \prod_{i=1}^{n} z_i^{k_i}\nonumber\\
&=& (1-s) \sum_{m=0}^\infty\combi{2m+r -1}{m} (1+s)^{-2m-r}
s^m
\\
&=& (1-s)\sum_{m=0}^\infty\combi{2m+r -1}{m} p^m q^{m+r}, \qquad
p=\frac{s}{1+s}=1-q\nonumber\\
&=& (1-s)q\sum_{m=0}^\infty P_0 ( \mathbb{S}_{2m+r-1}=1-r ).
\nonumber
\end{eqnarray}
Here, $P_a$ refers to the law of a downward-biased random walk $\{
\mathbb{S}_k\}$ with probability $p$ of going up at each step and
starting from $a$ at time zero.

Let $N_{k}$ denote the (almost surely finite) number $\sum
_{m=1}^\infty1\{ \mathbb{S}_m=k \}$. It then follows easily that
\[
\sum_{m=0}^\infty P_0 ( \mathbb{S}_{2m+r-1}=1-r )= E_0
N_{1-r} = 1 + E_{1-r} N_{1-r}= 1 + E_0 N_0=\frac{1}{1-\pi}.
\]
Here, $\pi$ is the return probability of the walk to zero, starting at
zero. Since $\pi$ is well known to be $2p$, one can substitute this
value into~\eqref{direct} and get
\[
\int_{\bndsimp} \varphi_z(y)\,dy=\frac{(1-s)q}{1-2p}=1
\]
since $s=p/q$. This completes the verification.

%s3 ###
\section{A skew-product decomposition result}\label{polar}

The BESQ family of measures is well known to be an additive family.
This can be utilized to embed multidimensional BESQ processes in a
\textit{measure-valued} BESQ process, as done by Shiga and Watanabe
\cite{shigawatanabe} and Pitman and Yor~\cite{besselbridge}. We
follow the statement and notation from~\cite{besselbridge}, Theorem 4.1.

Let $C[0,\infty)$ be the canonical space of continuous paths with the
usual topology. There exists a $C[0,\infty)$-valued process $ (
Y_x^\theta, \theta\ge0, x\ge0 )$\vadjust{\goodbreak} such that $Y_x^\theta$ has
law BESQ$_x^\theta$. Moreover, we have the additive decomposition
\[
Y_x^\theta= Y_x^0 + Y_0^\theta, \qquad  x\ge0,\ \theta\ge0,
\]
where $(Y^0_x, x\ge0)$ and $(Y_0^\theta, \theta\ge0)$ are
independent processes with stationary independent increments, each
having trajectories which are increasing and right-continuous with left
limits in $C[0,\infty)$.
In other words both $Y_0^\theta$ and $Y^0_x$ are independent
$C[0,\infty)$-valued L\'evy processes.

Now, fix any nonnegative $\theta_0$. Let $F$ be any distribution
function (increasing, right-continuous with left limits) on $[0,\theta
_0]$. Consider the $C[0,\infty)$-valued process $ (\Gamma_d,
0\le d \le\theta_0 )$, where
\[
\Gamma_d = Y^d_{F(d)}= Y^0_{F(d)} + Y^d_0 , \qquad  0\le d \le\theta_0.
\]

Let $\mathcal{P}([0,\theta_0])$ be defined as in Section \ref
{section1.2.3}. Given a realization of $\{\Gamma_d, 0 \le d \le
\theta_0\}$, one can construct a $\mathcal{P}([0,\theta_0])$-valued
process $\mu(t)$. For a fixed value of $t$ and a subinterval $(a,b]$
in $[0,\theta_0]$, it assigns a mass
\[
\mu(t)(a,b] = \frac{\Gamma_b(t) - \Gamma_a(t)}{\Gamma_{\theta_0}(t)}.
\]
This defines a probability measure uniquely, which we denote by the
following notation:
\[
\mu(t)(A) = \frac{1}{\Gamma_{\theta_0}(t)}\int_0^{\theta_0}
1(s\in A) \Gamma_{ds}(t)  \qquad \mbox{for all } A\in\mathcal
{B}([0,\theta_0]).
\]
We have the following skew-product decomposition result. Recall the
definition of the Fleming--Viot processes from Section~\ref{models}.

\begin{prop}\label{flemingviot}
Let $\sigma_0$ be the hitting time of zero for the process $\Gamma
_{\theta_0}$. There then exists an FV$[0,\theta_0]$ process $ \{
\nu(t), t\ge0 \}$, independent of $\Gamma_{\theta_0}$,
such that
\[
\mu(t) = \nu( 4 C_t ), \qquad  \mbox{where } C_t =
\int_0^{t} \frac{ds}{\Gamma_{\theta_0}(s)},\ t < \sigma_0.
\]
\end{prop}

The proof is essentially one step away from the simpler
finite-dimensional version that follows.

\begin{prop}\label{bestimechange}
Let $Z=(Z_1, \ldots,Z_n)$ be a vector of $n$ independent BESQ
processes, of dimensions $\theta_1, \ldots, \theta_n$.
Let $\zeta$ be the sum $\sum_{i=1}^n Z_i$, which is a BESQ of
dimension $\theta_0=\theta_1+ \cdots+\theta_n$. Assume that $\zeta
(0) >0$ and let
%
%e52 ###
\begin{equation}\label{whatissigma0}
\sigma_0 = \inf\{ t > 0\dvtx  \zeta(t) =0 \}.
\end{equation}

Then, there is an $n$-dimensional diffusion $\nu$, independent of
$\zeta$, and having law $J(\theta_1/2, \ldots, \theta_n/2)$, for which
%
%e53 ###
\begin{equation}\label{bestime}
Z(t) = \zeta(t) \nu( 4C_t ),\qquad C_t=\int_0^{t} \frac
{ds}{\zeta(s)},\   t < \sigma_0.
\end{equation}
\end{prop}

\begin{remark} The condition $t < \sigma_0$ is clearly
necessary to guarantee that the time change $C_t$ does not blow up. For
$n=2$ this result was noted by Warren and Yor in~\cite{warrenyor}; see
also the thesis by Goia~\cite{goiathesis09}.
\end{remark}

\begin{pf*}{Proof of Proposition~\ref{bestimechange}}
By our assumption, each $Z_i$ satisfies the following SDE:
\[
dZ_i(t)=\theta_i \,dt + 2\sqrt{Z_i(t)}\,d\beta_i(t), \qquad  i=1,2,\ldots,n.
\]
Let $R_i=Z_i/\zeta$. The SDE for $R_i$ for $t < \sigma_0$ can then be
found by It\^o's rule:

\begin{eqnarray*}
dR_i(t)&=& \zeta^{-1}\,dZ_i(t) + Z_i(t)\,d\zeta^{-1}(t) + d\langle
Z_i,\zeta^{-1} \rangle\\
&=& \zeta^{-1} \bigl[\theta_i \,dt + 2\sqrt{Z_i(t)}\,d\beta_i(t)
\bigr]+ Z_i(t) [-\zeta^{-2}\,d\zeta(t) + \zeta^{-3}\,d\langle\zeta \rangle
(t) ]\\
&&{} - 4Z_i\zeta^{-2}\,dt\\
&=& [\theta_i\zeta^{-1} - \theta_0 Z_i\zeta^{-2} +4Z_i\zeta
^{-2} - 4Z_i \zeta^{-2} ]\,dt\\
&&{}+ 2\zeta^{-1}\sqrt{Z_i(t)}\,d\beta_i(t) -2\zeta^{-2}Z_i(t)\sum
_{j=1}^n\sqrt{Z_j}\,d\beta_j\\
&=& \zeta^{-1} [\theta_i - \theta_0 R_i ]\,dt + 2\zeta
^{-1}\sqrt{Z_i(t)} [1- \zeta^{-1}Z_i(t) ]\,d\beta_i(t)\\
&&{}
-2\zeta^{-2}Z_i\sum_{j\neq i}\sqrt{Z_j}\,d\beta_j\\
&=&\zeta^{-1} [\theta_i - \theta_0 R_i ]\,dt + \zeta
^{-1/2}2\sqrt{R_i}\sum_{j=1}^n \bigl(1\{i=j\} - \sqrt{R_iR_j}
\bigr)\,d\beta_j(t).
\end{eqnarray*}
Define the sequence of local martingales
%
%e54 ###
\begin{equation}\label{howtodefmi}
dM_i(t)= \frac{\zeta^{-1/2}}{\sqrt{1- R_i}}\sum_{j=1}^n\bigl (1\{
i=j\} - \sqrt{R_iR_j} \bigr)\,d\beta_j(t)
\end{equation}
so that
%
%e55 ###
\begin{equation}\label{msde}
dR_i(t)=\zeta^{-1} [\theta_i - \theta_0 R_i ]\,dt + 2\sqrt
{{R_i}(1-R_i)}\,dM_i(t).
\end{equation}
However, since
\[
(1-R_i )^2 + \sum_{j\neq i}R_iR_j= (1-R_i )^2 +
(1-R_i )R_i=(1-R_i),
\]
it is guaranteed that $\langle M_i \rangle(t)=C_t$.

Let $\tau_u$ be the inverse of the increasing function $4C_t$, that
is, $\tau_u=\inf\{ t\dvtx  C_t \ge u/4 \}$. Let $\nu=(\nu
_1,\nu_2,\ldots,\nu_n)$ be the process obtained by time-changing $R$
by $\tau$. In other words, $\nu_i(u)=R_i(\tau_u)$. Applying this
time change to the SDE for $R_i$ in~\eqref{msde}, we get
%
%e56 ###
\begin{equation}\label{nuwt}
d\nu_i(t) = \tfrac{1}{4} [\theta_i - \theta_0 \nu_i ]\,dt
+ \sqrt{\nu_i(1-\nu_i)}\tw_i(t),\vadjust{\goodbreak}
\end{equation}
where $\tw_i$ is the Dambis--Dubins--Schwarz (DDS) Brownian motion
(see~\cite{yorbook}, page~181) associated with $M_i$. This is the SDE
for $J(\theta_1/2, \ldots, \theta_n/2)$ (see Section~\ref{models})
once we prove that the diffusion matrix is given by $\tilde{\sigma}$. To
compute it, note that
\[
\langle\nu_i, \nu_j \rangle(4C_t)= \langle R_i, R_j \rangle(t)=
\frac{4}{\zeta
(t)} \sqrt{R_i(t) R_j(t)} \bigl( 1\{i=j\} - \sqrt{R_iR_j} \bigr).
\]
Now, changing time by $\tau$, we immediately get $\langle\nu_i, \nu
_j \rangle=\tilde{\sigma}(i,j),$ as desired.

All that now remains to show is that the process $\nu$ above is
independent of $\zeta$. The SDE for $\zeta$ involves another martingale:
%
%e57 ###
\begin{equation}\label{sdezeta}
d\zeta(t) = \theta_0 \,dt + 2\sqrt{\zeta}\sum_{j=1}^n \sqrt
{R_j}\,d\beta_j(t)=\theta_0 \,dt + 2\sqrt{\zeta}\,d\beta^*(t).
\end{equation}
Here, $\beta^*$ is the local martingale $\int\sum\sqrt{R_j}\,d\beta
_j$, which is a standard Brownian motion by L\'evy's theorem (\cite{yorbook}, page 150). Note that
\[
d\langle\beta^*,M_i \rangle(t)= \frac{1}{\sqrt{1-R_i}} \biggl[\sqrt
{R_i} (1-R_i ) - \sqrt{R_i}\sum_{j\neq i}R_j \biggr]=0.
\]
Thus, by Knight's theorem (\cite{yorbook}, page 183), the DDS Brownian
motions of $(M_1, \ldots, M_n)$ and $\beta^*$ are independent. This
shows independence of $(\tw_1,\ldots,\break\tw_n)$ and $\beta^*$. It is
known (\cite{yorbook}, page 439) that $\zeta$ is a strong solution of
the SDE~\eqref{sdezeta}. Thus, from the independence proved above, it
follows that $\zeta$ is independent of the vector $(\tw_1,\ldots,\tw
_n)$ and hence $\nu$ in~\eqref{nuwt}. This completes the proof.
\end{pf*}

\begin{pf*}{Proof of Proposition~\ref{vsmproof}}
It will be useful for us now to analyze the time change $C_t$ in
Proposition~\ref{bestimechange}. Let us define $S(u)= \zeta(
\tau_u )$,
where $\tau$, used in the proof above, is the inverse of $4C_t$. Since
the derivative of $C_t$ with respect to $t$ is $1/\zeta(t)$, it
follows that
\[
\frac{d}{du} \tau_u = \frac{1}{4/\zeta(\tau_u)}= \frac{1}{4} S(u).
\]
In other words, $4\tau_u= \int_0^u S(t)\,dt$. Thus, if we define
$X_i(u)= Z_i(\tau_u)$ for $i=1,2,\ldots,n$, it follows that
\[
X_i(u) = Z_i ( \tau_u ), \qquad  \tau_u = \frac{1}{4}
\int_0^u S(t) \,dt,
\]
which is exactly the solution of $V(\theta_1/2, \ldots, \theta_n/2)$
described in the \hyperref[intro]{Introduction}.

The first part of Proposition~\ref{vsmproof} is now established. The
rest follows from known invariant distributions of Wright--Fisher
diffusions; see, for example,~\cite{ethkur81}.
\end{pf*}

\begin{pf*}{Proof of Proposition~\ref{flemingviot}}
Consider any finite sequence of Lebesgue measurable sets $\{ A_1, A_2,
\ldots, A_n\}$. By our construction of the L\'evy process of BESQ
processes, it follows that
%
%e58 ###
\begin{equation}\label{pathwise}
Z_i(t) = \int_0^{\theta_0} 1 ( s\in A_i ) \Gamma
_{ds}(t), \qquad  i=1,2,\ldots,n,
\end{equation}
are independent BESQ processes of respective dimensions $\theta_1,
\ldots, \theta_n$, where $\theta_i$ is the Lebesgue measure of
$A_i$. Note that the sum $\zeta=\Gamma_{\theta_0}$ is a BESQ process
of dimension $\theta_0=\theta_1 + \cdots+ \theta_n$.

By Proposition~\ref{bestimechange}, there is a Wright--Fisher
diffusion process $\nu(t)$ such that the time change relationship
\eqref{bestime} holds for all $t$ less than $\sigma_0$. As before,
let $\tau_u$ be the inverse of the increasing continuous function $4C_t$.

One can then define a measure on the $\sigma$-algebra generated by $\{
A_1, \ldots, A_n\},$ by defining
\[
\nu(u)(A_i) := \nu_i(u)= \frac{1}{\Gamma_{\theta_0}(\tau
_u)}Z_i(\tau_u), \qquad  i=1,2,\ldots,n.
\]
Note that in the pathwise construction~\eqref{pathwise}, the time
change is the same for all choices of $n$ and sets $\{ A_1, A_2, \ldots
, A_n\}$. Thus, it follows that the measure $\nu(u)$ is consistently
defined over any refinement of the sets $A_1, \ldots, A_n$. Moreover,
$\nu(u)$ is countably additive since it is derived from the measure
$\Gamma_{ds}(\tau_u)$. Thus, by a standard argument invoking the
Carath\'eodory extension theorem, a unique probability measure $\nu
(u)$ is established on the Borel sets in $[0,\theta_0]$. It is now
clear that the entire measure-valued process $\{ \nu(u), u\ge0\}$
satisfies all the defining properties of the Fleming--Viot model as
described in Section~\ref{models}.
\end{pf*}

%s3.1 ###
\subsection{Weights in a subset of the market}\label{section:submarket}

Thus far, our analysis has considered the entire vector of market
weights. It is often not possible to deal with all the stocks in a
single large market. Transactions are expensive and the different
market indices often concentrate on a chosen subcollection of stocks.

Thus, it is of interest to study the following problem. Suppose,
without loss of generality, we consider the first $m$ out of the total
of $n$ stocks in the equity market and define the process of submarket
weights as the vector
%
%e59 ###
\begin{equation}\label{submarket}
\tilde\mu= ( \tilde\mu_1, \ldots, \tilde\mu_m ), \qquad
\tilde\mu_i(t) = \frac{X_i(t)}{\tilde S(t)}, \qquad  \tilde
S(t)= \sum_{i=1}^m X_i(t).
\end{equation}
Can one describe the behavior of these submarket weights? The answer is
``yes,'' and the logic behind this relies on a self-recursive property
of the VSM models. Our next proposition makes this clear.

\begin{prop}
Consider the submarket weight vector $\tilde\mu$, as defined above.
There then exists a Wright--Fisher diffusion\vadjust{\goodbreak} $\tilde\nu$, independent of
the sum process $\sum_i \tilde\mu_i$, such that
\[
\tilde\mu_i(t) = \tilde\nu_i \biggl( \int_0^t \frac{du}{\sum_{j=1}^m
\tilde\mu_j(u)} \biggr), \qquad  i=1,\ldots,m,
\]
the equality holding for all $t$ until $\sum_i \tilde\mu_i$ hits zero.
\end{prop}

Since we have already shown that the market weights have the same law
as the Wright--Fisher models, we prove the proposition for the latter.
We take, without loss of generality, $m=n-1$, the case of a general $m$
following similar lines.

\begin{lemma}\label{nested}
Let $J=(J_1, \ldots, J_n)$ be the multidimensional diffusion $J(\delta
_1,\break \ldots,  \delta_n)$. Consider the process
\[
Y= \biggl( \frac{J_2}{1-J_1}, \frac{J_3}{1-J_1}, \ldots, \frac
{J_n}{1-J_1} \biggr),
\]
up to the stopping time $\tau_1=\inf\{ t\ge0\dvtx  J_1(t)=1
\}$.

There is then a diffusion $\tilde\nu$ which is $J(\delta_2, \delta
_3, \ldots, \delta_n)$, independent of $J_1$, such that
\[
Y(t)= \tilde\nu\biggl( \int_0^t \frac{ds}{1-J_1(s)} \biggr), \qquad
0\le t < \tau_1.
\]
\end{lemma}

\begin{pf}
Let $Z_1, Z_2, \ldots, Z_n$ be $n$ independent BESQ processes with
respective dimensions $2\delta_1, 2\delta_2, \ldots, 2\delta_n$.
Then, as shown in Proposition~\ref{bestimechange}, the process
%
%e60 ###
\begin{equation}\label{whatisji}
J_i(t)= \frac{Z_i(\tau_t)}{\zeta(\tau_t)},  \qquad \mbox{where }
\tau_t=\inf\biggl\{ u\ge0\dvtx  4\int_0^u \frac{ds}{\zeta(s)} \ge t
\biggr\},
\end{equation}
is distributed as $J(\delta_1, \ldots, \delta_n)$.

The proof utilizes the independence of the BESQ processes to derive the
stated result. We first claim that the time change $\int_0^t ds/\zeta
(s)$ grows to infinity almost surely. To see this, note that $\zeta$
is a BESQ process of dimension $d=\sum_{i=1}^n \delta_i$. When $d<
2$, the $\zeta$ process is recurrent, and thus the time change $\int
_0^t ds/\zeta(s)$ grows to infinity in finite time.

When $d > 2$, it is known (see, e.g.,~\cite{ferkarspt}, page 43) that
%
%e61 ###
\begin{equation}\label{timechangeinf}
\lim_{u\rightarrow\infty} \frac{1}{\log u}\int_0^u \frac
{ds}{\zeta(s)}= \frac{1}{d - 2}.
\end{equation}
Thus, the time change again grows to infinity with time $u$. The case
when $d=2$ can be sandwiched between the two cases above by using
stochastic comparison theorems for BESQ processes. Thus, the process
$J_i(t)$ in~\eqref{whatisji} has been constructed for all time $0\le t
< \infty$.

Now, let $\zeta_1=\sum_{i=2}^n Z_i$ be the sum of all the BESQ
processes except the first one. Exactly as before, there exists a
$(n-1)$-dimensional diffusion\vadjust{\goodbreak} $\tilde\nu=(\tilde\nu_2, \ldots,
\tilde\nu_n)$, independent of $\zeta_1$, with law $J(\delta_2,
\ldots, \delta_n)$ such that
\[
Z_{i}(t) = \zeta_1(t)\tilde\nu_i(4h_t), \qquad  h_t=\int_0^t \frac
{ds}{\zeta_1(s)},\ i=2,\ldots,n.
\]

Thus, for any $i > 1$, we get
\[
\frac{J_i(t)}{1-J_1(t)}= \frac{Z_i(\tau_t)}{\zeta_1(\tau_t)}=
\tilde\nu_{i} ( 4h(\tau_t) ) \qquad  \mbox{where }
h(\tau_t)= h_{\tau_t}.
\]

Let us now analyze the time change $h(\tau_t)$. We get
\[
4\frac{d}{dt} h(\tau_t)= 4h'(\tau_t) \tau'_t = \frac{\zeta(\tau
_t)}{\zeta_1(\tau_t)}=\frac{1}{1-J_1(t)}.
\]
The computation of $\tau_t'$ was carried out in the proof of
Proposition~\ref{vsmproof}.

Thus, we get the following description: for $i=2,\ldots,n$,
\[
\frac{J_i(t)}{1-J_1(t)}= \tilde\nu_{i} \biggl( \int_0^t \frac
{ds}{1-J_1(s)} \biggr) \qquad  \mbox{for all } t < \tau_1.
\]

Note that $\tilde\nu$ is independent of both $\zeta_1$ and $Z_1$. Thus,
$\tilde\nu$ is also independent of $J_1$. This completes the proof of
this result.
\end{pf}

%s4 ###
\section{Transition density of the market weights}\label{trandenmarket}

Finally, we combine all the results we have derived so far to obtain
the transition density for the market weights of the VSM model.

Our first step is to analyze the stopping time $\varsigma_a$ in
Proposition~\ref{vsmmarket}. To do this we return to the SDE \eqref
{vsmeq2} in the definition of the VSM model. As done in \cite
{ferkar05} we can express this SDE as
\[
dX_i(t)= \frac{\delta_i}{2} S(t)\,dt + \sqrt{X_i(t) S(t)} \,dW_i(t),
\qquad
i=1,2,\ldots,n.
\]
Summing over all the coordinates, we recover the SDE for the process
$S$ as
\[
d S(t) = \frac{d}{2} S(t) \,dt + \sqrt{S(t)} \sum_{i=1}^n \sqrt
{X_i(t)} \,dW_i(t) = \frac{d}{2} S(t) \,dt + S(t)\,d\beta(t),
\]
where $\beta$ is the local martingale $\int S^{-1/2}(t)\sum\sqrt
{X_i}(t) \,dW_i(t)$, which is a standard Brownian motion by L\'evy's
theorem (\cite{yorbook}, page 150).

Thus, $S$ is a geometric Brownian motion and can be alternatively
expressed as
\[
S(t) = S(0) \exp\bigl( (d-1)t/2 + \beta(t) \bigr).
\]
Thus, when $S(0)=s$, we get
\[
\varsigma_a= \inf\{ t\ge0\dvtx  \beta(t) + (d-1)t/2 \ge\log
(a/s) \}.
\]

The density of the $\varsigma_a$ is well known and can be found in the
book by Borodin and Salminen~\cite{handbookBM}. To simplify notation,
let us temporarily define
\[
\gamma:= (d-1)/2  \quad  \mbox{and} \quad  \rho= \log(a/s).
\]
Note that $\rho$ is assumed to be positive.

The density of $\varsigma_a$ is then given by
%
%e62 ###
\begin{eqnarray}\label{lawoftaua}
P ( \varsigma_a \in dt )&=& \frac{\log(a/s)}{\sqrt{2\pi
t^3}}\exp\biggl( -\frac{(\log(a/s)-\gamma t)^2}{2t} \biggr)\,dt\nonumber
\\[-8pt]
\\[-8pt]
&=&\frac{\rho}{\sqrt{2\pi t^3}}\exp\biggl( -\frac{(\rho-\gamma
t)^2}{2t} \biggr)\,dt.
\nonumber
\end{eqnarray}

Recall now from Proposition~\ref{vsmproof} that the process $S$ is
independent of the market weights $\mu$. Thus, $\mu$ and $\varsigma
_a$ are also independent. Suppose we denote the transition density
function of $\mu$ at time $t$ by $p(t,\xi,y)$, where the\vspace*{1pt} initial
position $\xi$ and the terminal position $y$ are both elements of
$\bndsimp$. Then, by the independence, it follows that
%
%e63 ###
\begin{equation}\label{trandenexit}
\int_0^t p(t,\xi,y) P ( \varsigma_a \in dt ) = \varphi
_x(y), \qquad  x=s\xi=ae^{-\rho}\xi,
\end{equation}
where $\varphi$ is the exit density computed in Proposition~\ref{vsmmarket}.

Note that the above integral transform~\eqref{trandenexit} has a
single parameter $\rho$ if we fix $a=1$. Equation~\eqref{trandenexit} is an
integral transform. If this transform can be inverted, we can recover
$p(t,\xi,y)$ from $\varphi$. As we show below, this integral
transform is nothing but a Laplace transform in disguise.

To see this, note that for any function $h(t)$ (keeping $a=1$), we get
\begin{eqnarray*}
\int_0^{\infty}  h(t)P ( \varsigma_1 \in dt ) &=& \int
_0^\infty h(t) \frac{\rho}{\sqrt{2\pi t^3}} \exp\biggl( -\frac
{(\rho-\gamma t)^2}{2t} \biggr)\,dt\\
&=& \frac{\rho}{\sqrt{2\pi}}\int_0^\infty h(t)t^{-3/2} \exp\biggl(
-\frac{(\rho^2 + \gamma^2 t^2 - 2\rho\gamma t)}{2t} \biggr)\,dt\\
&=& \frac{\rho e^{\rho\gamma}}{\sqrt{2\pi}}\int_0^\infty
h(t)t^{-3/2} e^{-\gamma^2 t/2} \exp\biggl( -\frac{\rho^2}{2t}
\biggr)\,dt.
\end{eqnarray*}

If we now change the variable from $t$ to $u=1/t$, we get
%
%e64 ###
\begin{equation}\label{finaleq}
\Upsilon(\rho)=\int_0^{\infty} h(t)P ( \varsigma_1 \in dt
) = \frac{\rho e^{\rho\gamma}}{\sqrt{2\pi}} \int_0^\infty
g(u) e^{-\rho^2u/2}\,du,
\end{equation}
where the function $g$ is defined by
\[
g(u)= h(1/u)u^{3/2} e^{-\gamma^2/(2u)}u^{-2} = h(1/u) u^{-1/2}
e^{-\gamma^2/(2u)}.
\]
Thus, we get that
%
%e65 ###
\begin{equation}\label{twolaplace}
\sqrt{2\pi}\Upsilon(\rho)\rho^{-1}e^{-\rho\gamma} = \Lambda
(g) ( \rho^2/2 ),
\end{equation}
where $\Lambda(g)(\cdot)$ represents the Laplace transform of $g$. In
other words, the function $g$ (hence $h$) can be recovered by inverting
the Laplace transform.

We are now going to apply the preceding analysis to the function
$h(t)=p(t,\xi,y)$ for fixed values of $\xi$ and $y$ in the set
$\bndsimp$. In that case, from~\eqref{trandenexit} and the formula
\eqref{mwtden} (taking $a=1$ and $x=e^{-\rho}\xi$), we get that
\begin{eqnarray*}
\Upsilon(\rho)&=& ( 1 - e^{-\rho} )\sum_{m=0}^{\infty}
\frac{\Gamma(2m + d)}{m!\Gamma(m + d)} (1+ e^{-\rho}
)^{-2m-d} \\
&&{}\times\sum_{k\ge0: k_1 + \cdots+ k_n=m} \combi{m}{k_1\cdots
k_n} \prod_{i=1}^{n} (e^{-\rho}\xi_i)^{k_i} \operatorname{Dir}(y;
k+\delta), \qquad  y\in\bndsimp.
\end{eqnarray*}
Simplifying slightly, we get
\begin{eqnarray*}
\Upsilon(\rho)&=& ( 1 - e^{-\rho} )\sum_{m=0}^{\infty}
\frac{\Gamma(2m + d)}{m!\Gamma(m + d)} (1+ e^{-\rho}
)^{-2m-d}  e^{-m\rho}\\
&&{}\times\sum_{k\ge0: k_1 + \cdots+ k_n=m} \combi
{m}{k_1\cdots k_n} \prod_{i=1}^{n} (\xi_i)^{k_i} \operatorname{Dir}(y;
k+\delta), \qquad  y\in\bndsimp.
\end{eqnarray*}
In particular, this nice series representation allows us to take the
inverse Laplace transform inside the infinite sum and obtain the final formula:
\begin{eqnarray*}
g(u)&=& \sum_{m=0}^{\infty} \frac{\Gamma(2m + d)}{m!\Gamma(m + d)}
\varrho_m(u) \\
&&\hphantom{\sum_{m=0}^{\infty}}
{}\times\sum_{k: k_1 + \cdots+ k_n=m} \combi{m}{k_1\cdots k_n}
\prod_{i=1}^{n} (\xi_i)^{k_i} \operatorname{Dir}(y; k+\delta), \qquad  y\in
\bndsimp.
\end{eqnarray*}
Here, $\varrho_m(u)$ is defined by the Laplace transform formula
\[
\Lambda( \varrho_m ) ( \rho^2/2 )= \sqrt
{2\pi}\rho^{-1}e^{-(m+\gamma)\rho} ( 1 - e^{-\rho}
) (1+ e^{-\rho} )^{-2m-d}, \qquad  m=0,1,\ldots.
\]
Changing the variable back to $t=1/u$, we obtain
\begin{eqnarray*}
h(t)&=& h(1/u)= u^{1/2}e^{\gamma^2/(2u)}g(u)\\
&=&\sum_{m=0}^{\infty}
\frac{\Gamma(2m + d)}{m!\Gamma(m + d)} u^{1/2}e^{\gamma
^2/(2u)}\varrho_m(u)\\
&&\hphantom{\sum_{m=0}^{\infty}}
{} \times\sum_{k: k_1 + \cdots+ k_n=m} \combi{m}{k_1\cdots k_n}
\prod_{i=1}^{n} (\xi_i)^{k_i} \operatorname{Dir}(y; k+\delta)\\
&=&\sum_{m=0}^{\infty} \frac{\Gamma(2m + d)}{m!\Gamma(m + d)}
b_m(t)\sum_{k: k_1 + \cdots+ k_n=m} \combi{m}{k_1\cdots k_n} \prod
_{i=1}^{n} (\xi_i)^{k_i} \operatorname{Dir}(y; k+\delta).
\end{eqnarray*}
Here, the coefficients $b_m$ are given by [see~\eqref{finaleq}]
\[
\int_0^{\infty} b_m(t) t^{-3/2} e^{-\gamma^2t/2} \exp\biggl(
-\frac{\rho^2}{2t} \biggr)\,dt= \Lambda(\varrho_m)(\rho^2/2).
\]
This establishes Proposition~\ref{tranmarket}.

\section*{Acknowledgments} I am grateful to Chris Burdzy, Zhen-Qing
Chen and Jim Pitman for valuable comments and discussion. I also wish
to thank the two referees for an extremely careful reading of the
manuscript and for suggesting numerous points for improvement.

% imsref loaded by smiklovaite, 2010-11-10 08:38:52
%

\printaddresses

\end{document}